\documentclass[11pt,a4paper,reqno]{amsart} 

\usepackage{setspace}
\usepackage[T1]{fontenc}
\usepackage[utf8]{inputenc}
\usepackage{lmodern}
\usepackage{graphicx}
\usepackage{amsthm} 
\usepackage{amsfonts}
\usepackage{amsmath} 
\usepackage{amssymb}
\usepackage{latexsym}
\usepackage{pdfpages}
\usepackage{subfig}
\usepackage{tikz-cd}
\usepackage{mdframed}
\usepackage{setspace}
\usepackage{comment}
\usepackage{cancel}
\usepackage{ulem}

\usepackage[pdfencoding=auto, psdextra]{hyperref}

\usepackage{indentfirst}

\graphicspath{{FigureD/}}

\newtheorem{Theorem}{Theorem}[section]
\newtheorem{Lemma}[Theorem]{Lemma}
\newtheorem{Proposition}[Theorem]{Proposition}
\newtheorem{Corollary}[Theorem]{Corollary}

\theoremstyle{definition}
\newtheorem{Definition}[Theorem]{Definition}
\newtheorem{Remark}[Theorem]{Remark}

\newtheorem{Example}[Theorem]{Example}
\newtheorem{Notation}[Theorem]{Notation}

\newcommand{\PLR}{\mathcal{P}^{LR}_k(\Omega)}
\newcommand{\PLRb}{\mathcal{P}^{LR}_k(\Omega')}

\newcommand{\Zd}{\mathbb{Z}[\delta]}

\newcommand{\tb}{\vartriangle}
\newcommand{\weight}{\mathsf{w}}
\newcommand{\PLRn}{\widehat{\mathcal{P}}^{LR}_{n+2}(\Omega)}
\newcommand{\PLRnb}{\widehat{\mathcal{P}}^{LR}_{n+2}(\Omega')}

\newcommand{\PLRk}{\widehat{\mathcal{P}}^{LR}_{k}(\Omega)}
\newcommand{\PLRkb}{\widehat{\mathcal{P}}^{LR}_{k}(\Omega')}

\newcommand{\ab}{{\bf a}}
\newcommand{\TLR}{T^{LR}_k(\Omega)}

\newcommand{\lob}{\mathcal{L}_\bullet}
\newcommand{\lot}{\mathcal{L}_\bullet^{\tb}}
\newcommand{\loc}{\mathcal{L}_\circ}
\newcommand{\locb}{\mathcal{L}_\bullet^{\circ}}

\newcommand{\TLD}{\mathrm{TL}(\widetilde{D}_{n+2})}
\newcommand{\Bn}{\widetilde{B}_{n+1}}
\newcommand{\Dn}{\widetilde{D}_{n+2}}
\newcommand{\Cn}{\widetilde{C}_{n}}
\newcommand{\DB}{\mathbb{D}(\widetilde{B}_{n+1})}
\newcommand{\DD}{\mathbb{D}(\widetilde{D}_{n+2})}

\newcommand{\du}{\mathbf{d_1}}
\newcommand{\dd}{\mathbf{d_2}}

\newcommand{\db}{\mathtt{ad}(\widetilde{B}_{n+1})}
\newcommand{\dbD}{\mathtt{ad}(\widetilde{D}_{n+2})}

\newcommand{\TL}{\operatorname{TL}}
\newcommand{\FC}{\operatorname{FC}}

\newcommand{\ALT}{\operatorname{ALT}}
\newcommand{\LP}{\operatorname{LP}}
\newcommand{\RP}{\operatorname{RP}}
\newcommand{\LRP}{\operatorname{LRP}}
\newcommand{\PZZ}{\operatorname{PZZ}}

\newcommand{\cp}{\mathrm{cp}}

\newcommand{\HZ}{H_{Z_1}}
\newcommand{\HZZ}{H_{Z_2}}

\definecolor{springgreen}{rgb}{0.0, 1.0, 0.5}
\definecolor{magenta(process)}{rgb}{1.0, 0.0, 0.56}

\begin{document}
\title[Diagrammatic representations of affine types $\widetilde{D}$ and $\widetilde{B}$]{Diagrammatic representations of Generalized Temperley--Lieb algebras of affine types $\widetilde{D}$ and $\widetilde{B}$}

\author[R. Biagioli]{Riccardo Biagioli}

\address{R. Biagioli, E. Sasso: Dipartimento di Matematica, Universit\`a di Bologna, Piazza di Porta San Donato 5, 40126 Bologna, Italy}
\email{riccardo.biagioli2@unibo.it, elisa.sasso2@unibo.it}

\author[G. Fatabbi]{Giuliana Fatabbi}

\address{G. Fatabbi: Dipartimento di Matematica, Universit\`a degli Studi di Perugia, Via Vanvitelli 1, 06123 Perugia, Italy}
\email{giuliana.fatabbi@unipg.it}

\author[E. Sasso]{Elisa Sasso}

\begin{abstract}
Let $(W,S)$ be an affine Coxeter system of type $\Gamma$, with $\Gamma$ equal to $\widetilde{D}$ or $\widetilde{B}$, and $\TL(\Gamma)$ the corresponding generalized Temperley--Lieb algebra. In this paper we define an infinite dimensional associative algebra made of decorated diagrams that is isomorphic to $\TL(\Gamma)$.  Moreover, we describe an explicit basis for such an algebra consisting of special decorated diagrams that we call admissible. Such basis is in bijective correspondence with the classical monomial basis of the generalized Temperley--Lieb algebra indexed by the fully commutative elements of $W$. 
\end{abstract}
\date{\today}


\maketitle

%\tableofcontents

%%%%%%%%%%%%%%%%%%%%%%%%%%%%%%
\section*{Introduction}\label{sec:intro}
%%%%%%%%%%%%%%%%%%%%%%%%%%%%%%

The Temperley--Lieb (TL) algebra is a very classical mathematical object studied in algebra, combinatorics, statistical mechanics and mathematical physics, introduced by Temperley and Lieb in 1971 \cite{TemperleyLieb}. Kauffman \cite{Kauffman} and Penrose \cite{Penrose} showed that the TL algebra can be realized as a diagram algebra, that is an associative algebra with a basis given by certain diagrams on the plane. On the other hand, Jones presented the TL algebra in terms of abstract generators and relations. In \cite{Jones-hecke}, he also showed that this algebra occurs naturally as a quotient of the Hecke algebra of type $A$. The realization of the TL algebra as a Hecke algebra quotient was generalized by Graham in \cite{Graham}. He defined the so-called \textit{generalized Temperley--Lieb algebra} TL($\Gamma$) for a Coxeter system of arbitrary type $\Gamma$ and showed that it has a monomial basis indexed by the fully commutative elements $\FC(\Gamma)$ of the underlying Coxeter group. Over the years, some diagrammatic representations for $\TL(\Gamma)$ have been found. More precisely, Green defined a diagram calculus in finite Coxeter types $B$, $D$, $E$ and $H$, respectively in \cite{Green-general}, \cite{Green_TLE} and \cite{Green-cellular}. For affine types, in \cite{FanGreen_Affine} Fan and Green provided a realization of TL($\widetilde{A}$) as a diagram algebra on a cylinder and, more recently, in \cite{ErnstDiagramI, ErnstDiagramII}, Ernst represented $\TL(\widetilde{C}$) as an algebra of decorated diagrams.
\smallskip

The aim of this paper is to give diagrammatic representations for the two remaining affine families, namely types $\widetilde{D}$ and $\widetilde{B}$. Taking inspiration from Ernst's work on type $\widetilde{C}$,  we consider two subsets of the decorated diagrams introduced in~\cite{ErnstDiagramI}, we endow them  with an algebra structure, and then we define, modulo some relations, two new quotient algebras, denoted by $\DD$ and $\DB$. Our main results show that these algebras are actually isomorphic to $\TL(\Dn)$ and $\TL(\Bn)$, respectively. 
\smallskip

In order to do so, we first give an explicit description of a special family of irreducible decorated diagrams, called \textit{admissible}, which turns out to be a linear basis of $\DD$. We then introduce a notion of length for any admissible diagram that measures in a certain sense the distance of the decorations from the left and the right side of the diagram. Finally, we define procedures that allow us to factorize any decorated diagram in terms of the so-called simple diagrams. The most relevant, called the \textit{cut and paste operation}, transforms an admissible diagram of a certain type in another one of lower length by acting on some of its edges. Now it remains to prove that our algebra of decorated diagrams is a faithful representation of $\TL(\Dn)$. 
\smallskip

 Several strategies have been used to prove the faithfulness of the known diagrammatic representations of TL algebras of arbitrary type. For instance, in type $\widetilde{C}$ two different proofs of injectivity exist: the original one due to Ernst~\cite{ErnstDiagramII} uses non-cancellable elements in Coxeter groups, while more recently in \cite{BFC}, an algorithmic proof based on a decomposition of the heaps associated with fully commutative elements is proposed. In this paper, we first prove the injectivity of the representation in type $\widetilde{D}$ by using an inductive argument similar to that of Fan and Green in type $\widetilde{A}$ \cite{FanGreen_Affine}. Then, following a suggestion of Green \cite{RG}, we define ad hoc injections between $\FC(\Bn)$ and $\FC(\Dn)$ and between $\DB$ and $\DD$, from which we deduce that $\DB$ is a faithful diagrammatic representation of $\TL(\Bn)$.
\smallskip

This paper is organized as follows. In Section~\ref{sub:fullycomm} we describe the necessary background in Coxeter group theory, heap theory and the presentations of the TL algebras we use. In Section~\ref{sec:undecorated} we recall Ernst's definition of decorated diagrams and we introduce our quotient subalgebra $\DD$ by defining appropriate relations on the set of decorations. In Section~\ref{sec:admissible} we define the admissible diagrams and we also introduce a notion of length for them. In Section~\ref{sec:simple edges-cp} we define the cut and paste operation. Then in Section~\ref{sec:isom} we prove that the two algebras $\TL(\Dn)$ and $\DD$ are isomorphic. In Section~\ref{Green}, we explain how all the material developed in type $\widetilde{D}$ needs to be changed in order to define a diagrammatic representation of $\TL(\Bn)$ and we prove the faithfulness of this representation. Finally, in Section~\ref{sec:ultimo} we discuss some further developments related to our work.
%%%%%%%%%%%%%%%%%%%%%%%%%%%%%%%%
\section{Fully commutative elements and generalized Temperley--Lieb algebras}\label{sub:fullycomm}
%%%%%%%%%%%%%%%%%%%%%%%%%%%%%%%%

Let $M$ be a square symmetric matrix indexed by a finite set $S$, satisfying $m_{ss}=1$ and, for $s\neq t$, $m_{st}=m_{ts}\in\{2,3,\ldots\}\cup\{\infty\}$. The \textit{Coxeter group} $W$ associated with the  \textit{Coxeter matrix} $M$ is defined by generators $S$ and relations $(st)^{m_{st}}=1$ if $m_{st}<\infty$. These relations can be rewritten more explicitly as $s^2=1$ for all $s$, and \[\underbrace{sts\cdots}_{m_{st}}  = \underbrace{tst\cdots}_{m_{st}},\]  where $m_{st}<\infty$, the latter being called \textit{braid relations}. When $m_{st}=2$, they are simply \textit{commutation relations} $st=ts$; when $m_{st}=\infty$, there are no relations between $s$ and $t$.

The \textit{Coxeter graph} $\Gamma$ associated to the \textit{Coxeter system} $(W,S)$ is the graph with vertex set $S$ and, for each pair $\{s,t\}$ with $m_{st}\geq 3$, an edge between $s$ and $t$ labeled by $m_{st}$. When $m_{st}=3$ the edge is usually left unlabeled since this case occurs frequently. Therefore nonadjacent vertices correspond precisely to commuting generators.
 
\begin{figure}[h]
\centering
\includegraphics[width=0.9\linewidth]{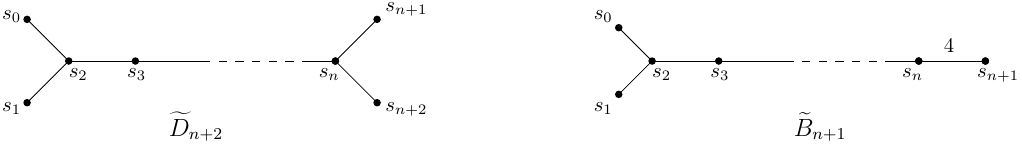} 
\caption{Coxeter graphs of type $\Dn$ and $\Bn$.}
\label{BDtilde}
\end{figure}

For $w\in W$, the \textit{length} of $w$, denoted by $\ell(w)$, is the minimum length $l$ of an expression $w=s_1\cdots s_l$ with $s_i\in S$. The expressions of length $\ell(w)$ are called \textit{reduced} and the set of all the reduced expressions of $w$ is denoted by $\mathcal{R}(w)$. Expressions are usually denoted with bold letters. 
Given a reduced expression $\mathbf{w}$ of $w$, we say that a \textit{factor} of ${\bf w}$ is a subsequence of letters of ${\bf w}$ in consecutive positions.

A fundamental result in Coxeter group theory, sometimes called the \textit{Matsumoto property}, states that any reduced expression of $w$ can be obtained from any other reduced expression of $w$ by using only braid relations (see for instance~\cite{Humphreys}). 

We define an equivalence relation in $\mathcal{R}(w)$ as the reflexive transitive closure of the relation: given $\mathbf{w}, \mathbf{w'}\in \mathcal{R}(w)$, $\mathbf{w}\sim\mathbf{w'}$ if and only if $\mathbf{w'}$ can be obtained from $\mathbf{w}$ by a single commutation relation. Equivalence classes of this relation are called \textit{commutation classes}.

\begin{Definition}
	An element $w\in W$ is \textit{fully commutative} ($\FC$) if $w$ has exactly one commutation class. The set of all fully commutative elements of $W$ is denoted by $\FC(\Gamma)$, where $\Gamma$ is the associated Coxeter graph.
\end{Definition}

\begin{Example}\label{ex:fc}
    The element $w=s_3s_5s_2s_4s_0s_1s_3s_2$ is in  $\FC(\widetilde{D}_6)$ and $u=s_4s_2s_1s_3s_2$ in $\FC(\widetilde{D}_4)$, while the element $v=s_1s_2s_1s_3$ in $\widetilde{D}_4$ is not fully commutative.
\end{Example}
 
The next proposition, due to Stembridge, characterizes the FC elements of type $\Gamma$ and it is useful to test whether a given element is FC or not.

\begin{Proposition} [Stembridge \cite{Stem-fullycomm}, Prop. 2.1]
\label{prop:caracterisation_fullycom}
An element $w\in \FC(\Gamma)$ if and only if for all $s,t$ such that $3\leq m_{st}<\infty$, there is no reduced expression of $w$ that contains the factor $\underbrace{sts\cdots}_{m_{st}}$.
\end{Proposition}

 The concept of heap helps to capture the notion of full commutativity. We briefly describe a way to define the above mentioned heap and its relations with full commutativity; for more details see for instance~\cite{BJN_FC} and the references cited there.

Let $(W,S)$ be a Coxeter system with Coxeter graph $\Gamma$, and fix a reduced expression $\mathbf{w}=s_{a_1}\cdots s_{a_l}$ with $s_{a_j} \in S$. Define a partial ordering $\prec$ on the index set $\{1,\ldots, l\}$ as follows: set $i\prec j$ if $i<j$ and $s_{a_i}$, $s_{a_j}$ do not commute, and extend it by transitivity. Observe that if $i<j$ and $s_{a_i} = s_{a_j}$, then $i\prec j$ by transitivity and the fact that $\mathbf{w}$ is reduced.
We denote by $H({\bf w})$ the triple $(\{1, \ldots, l\}, \prec, \varepsilon)$, where $\varepsilon:i\mapsto s_{a_i}$ is the labeling map and we call $H(\mathbf{w})$ a \textit{labeled heap of type $\Gamma$} or simply a \textit{heap}. Moreover, if $F$ is a subset of $\{1, \ldots, l\}$ and $\prec'$ and $\varepsilon'$ are the restrictions of $\prec$ and $\varepsilon$ to $F$, then $(F,\prec', \varepsilon')$ is called a \textit{subheap} of $H({\bf w})$.  In particular, given a heap $H=H({\bf w})$ and a subset $I\subset S$, we denote by $H_{I}=(F, \prec', \varepsilon')$ the subheap of $H$ where $F$ is the set of all elements of $\{1, \ldots, l\}$ with labels in $I$ (see~\cite[\S 2]{ViennotHeaps}).

Heaps are well-defined up to commutation classes~\cite{ViennotHeaps}; that is, if ${\bf w}, {\bf w'} \in \mathcal{R}(w)$ belong to the same commutation class, then the corresponding labeled heaps are isomorphic. This means that there exists a poset isomorphism between the heaps that preserves the labels. Therefore, when $w$ is FC  we can define $H(w):=H(\mathbf{w})$, where ${\bf w}$ is any reduced expression for $w$. Heaps of this form will be called \textit{FC heaps}. Moreover, if $H(w)$ is a FC heap then its linear extensions are in bjiection with $\mathcal{R}(w)$; see \cite[Proposition 2.2]{Stem-fullycomm}.
\smallskip

In the Hasse diagram of $H({\bf w})$, vertices with the same labels are drawn in the same column. Moreover, as in \cite{ErnstDiagramI}, we draw heaps from top to bottom, namely the vertices on the top of $H({\bf w})$, correspond to generators occurring on the left of ${\bf w}$. We often identify vertices with their labels.

\begin{Example}
The heaps of $w=s_3s_5s_2s_4s_0s_1s_3s_2$, $u=s_4s_2s_1s_3s_2$ and $v=s_1s_2s_1s_3$ (which is not fully commutative) of Example \ref{ex:fc} are represented in Figure~\ref{heapD}. The subheap of $H(w)$ corresponding to the subset $\{s_2,s_3\}$ is dashed.
\end{Example}

\begin{figure}[!ht]
\begin{center}
\includegraphics[height=25mm, keepaspectratio]{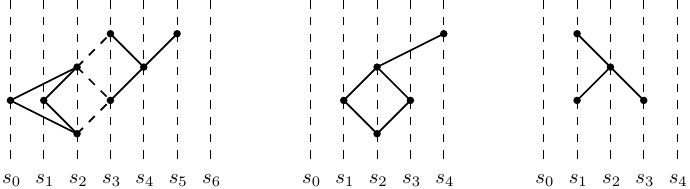}
\caption{The heap $H(w)$, the subheap $H(w)_{\{s_2,s_3\}}$ dashed, $H(u)$ and $H(v)$.}\label{heapD}  
\end{center}
\end{figure}

\begin{Definition}
\label{defi:alternating}
Let $(W,S)$ be a Coxeter system of type $\Gamma$, $w \in \FC(\Gamma)$, and $H:=H(w)$. We say that $H$ is \textit{alternating} if for each pair of non-commuting generators $s,t$ in $S$, the chain $H_{\{s,t\}}$ has alternating labels $s$ and $t$ from top to bottom.
\end{Definition}

Both examples in Figure~\ref{heapD} are alternating. Note that if $H(w)$ is alternating, then any reduced expression $\mathbf{w}$ of $w$ is \textit{alternating} in the sense that for each  non-commuting generators  $s,t \in S$, the occurrences of $s$ alternate with those of $t$ in $\mathbf{w}$. In this case we also say that $w \in \FC(\Gamma)$ is alternating. 
\smallskip

In this work, we focus on the Coxeter systems of types $\Dn$ and $\Bn$. Their Coxeter graphs are depicted in Figure \ref{BDtilde}, and they completely determine the corresponding Coxeter systems. For more details and combinatorial characterization of such groups we refer to \cite[\S 8]{BB}.

As recalled in the introduction, FC elements in a Coxeter group $W$ index a basis of the so-called generalized Temperley--Lieb algebra of $W$. The following are the presentations of $\TL(\Dn)$ and $\TL(\Bn)$ given by Green in \cite[Proposition 2.6]{Green-Star} that will be used in our work as definitions.

\begin{Definition}\label{def:tl-algebras}
The \textit{Temperley--Lieb algebra of type $\Dn$}, denoted by TL($\Dn$), is the $\Zd$-algebra generated by $\{b_0, b_1, \ldots, b_{n+2}\}$ with defining relations:

\begin{enumerate}
\item[(d1)] $b_i^2=\delta b_i$ for all $i\in \{0,\ldots, n+2\}$;
\item[(d2)] $b_i b_j=b_jb_i$ if $s_i$ and $s_j$ are not adjacent nodes in the Coxeter graph of type $\Dn$;
\item[(d3)] $b_i b_j b_i=b_i$ if $s_i$ and $s_j$ are adjacent nodes in the Coxeter graph of type $\Dn$.
\end{enumerate}
\end{Definition}

\begin{Definition}\label{def:tlB-algebra}
The \textit{Temperley--Lieb algebra of type $\Bn$}, denoted by TL($\Bn$), is the $\Zd$-algebra generated by $\{b_0, b_1, \ldots, b_{n+1}\}$ with defining relations:
\begin{enumerate}
\item[(b1)] $b_i^2=\delta b_i$ for all $i\in \{0,\ldots, n+1\}$;
\item[(b2)] $b_ib_j=b_jb_i$ if $s_i$ and $s_j$ are not adjacent nodes in the Coxeter graph of type $\Bn$; 
\item[(b3)] $b_i b_j b_i=b_i$ if $s_i$ and $s_j$ are adjacent nodes in the Coxeter graph of type $\Bn$, and $\{i,j\}\neq\{n,n+1\}$;
\item[(b4)] $b_i b_j b_i b_j=2b_ib_j$ if $\{i,j\}=\{n,n+1\}$.
\end{enumerate}
\end{Definition}

Let $\Gamma$ be $\Dn$ or $\Bn$. For any $s_{i_1}\cdots s_{i_k}$ reduced expression of $w$ in $\FC(\Gamma)$, we set
$$b_w=b_{i_1} \cdots \ b_{i_k}.$$
It is easy to see that $b_w$ does not depend on the chosen reduced expression of $w$. Graham proved that the set $\{b_w \mid w\in \FC(\Gamma)\}$ is a basis for $\TL(\Gamma)$, which is usually called the \textit{monomial basis} (see also \cite[Proposition 2.4]{Green-Star}).
\smallskip

From now on we will focus on the affine Coxeter group of type $\Dn$ whose Coxeter graph is depicted in Figure~\ref{BDtilde}. 

\begin{Notation}
By the description given above, a heap of $w \in \FC(\Dn)$ should be represented in a Hasse diagram with $n+3$ columns labeled by $s_0, s_1, s_2, \ldots, s_{n+2}$, as we did in Figure~\ref{heapD}. However, when dealing with Coxeter systems of type $\Dn$, we will represent heaps in $n+1$ columns with the following criterion. Both vertices with label $s_0$ and $s_1$ (respectively $s_{n+1}$ and $s_{n+2}$) will be drawn in the first (respectively last) column of $H(w)$. Moreover, when $s_0s_1$ (respectively $s_{n+1}s_{n+2}$) occurs in a reduced expression of $w$, we represent the vertices associated to these $s_0$ and $s_1$ (respectively $s_{n+1}$ and $s_{n+2}$) with a single mark point labeled $s_0s_1$ (respectively $s_{n+1}s_{n+2}$). Vertices labeled by $s_0$, $s_1$ and $s_0s_1$ will be called \textit{1-elements} while vertices labeled by $s_{n+1}$, $s_{n+2}$ and $s_{n+1}s_{n+2}$ will be called \textit{(n+1)-elements} and we will explicitly write their labels in the heap representation. For instance from now on, the two heaps in Figure~\ref{heapD} will be depicted as in Figure~\ref{heapDmod}.
\end{Notation}

\begin{figure}[hbtp]
\centering
\includegraphics[height=25mm, keepaspectratio]{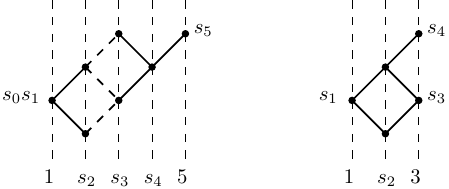}
\caption{The new representations of $H(w)$ and $H(u)$ of Figure~\ref{heapD}.}
\label{heapDmod}
\end{figure}

Fully commutative heaps of type $\Dn$ have been classified in  \cite[\S 3.2]{BJN_FC}. We present now a slightly modified version of the original classification \cite[Definition 3.9]{BJN_FC}.
\smallskip

If $i\in\{2, \ldots, n\}$, a \textit{peak} is a heap of the form:
\begin{align*}
    P_{\rightarrow}(s_i)&:=H(s_is_{i+1}\cdots s_ns_{n+1}s_{n+2}s_n\cdots s_{i+1}s_i) \mbox{ or }\\
    P_{\leftarrow}(s_i)&:=H(s_is_{i-1}\cdots s_2 s_0s_1 s_2\cdots s_{i-1}s_i).
\end{align*}

If $H$ is a heap of type $\Dn$ and $i\in\{2, \ldots, n\}$, then for brevity, we denote by $H_{\{\leftarrow s_i\}}$ (respectively, $H_{\{\rightarrow s_i\}}$) the subheap of $H$ with labels in the set $\{s_0s_1, s_0, s_1, s_2, \ldots, s_{i-1}, s_i\}$ (respectively, $\{s_i, s_{i+1}, \ldots, s_{n+1}, s_{n+2}, s_{n+1}s_{n+2}\}$). 

\begin{Definition}\label{family}
We define the following five families of FC heaps of type $\Dn$.

\begin{itemize}
\item[\textbf{(PZZ)}] \textit{Pseudo Zigzags}. $H\in$ (PZZ) if $H=H(w)$ where $w$ has a reduced expression with at least one occurrence of both $s_0s_1$ and $s_{n+1}s_{n+2}$ and that is a factor of an expression in the commutation class of the element $\left(s_0s_1 s_2\cdots s_{n}s_{n+1}s_{n+2}s_{n}\cdots s_2\right)^k$ for some integer $k\geq 1$.
\smallskip

\item[\textbf{(ALT)}] \textit{Alternating}. $H\in$ (ALT) if it is alternating in the sense of Definition \ref{defi:alternating}, where $\Gamma$ is of type $\Dn$, but it is different from $\HZ$ and $\HZZ$ in Figure~\ref{heapspart}. Moreover, if $H$ contains two or more 1-elements (respectively ($n+1$)-elements), those must be $s_0$ and $s_1$ (respectively $s_{n+1}$ and $s_{n+2}$) and have to alternate. 
\smallskip

\item[\textbf{(LP)}] \textit{Left-Peaks}. $H\in$ (LP) if there exists $j_{\ell}\in \{2, \ldots, n\}$ such that:
\begin{itemize}
\item[(1)] $H_{\{\leftarrow s_{j_{\ell}}\}}=P_{\leftarrow}(s_{j_{\ell}})$;
\item[(2)] There is no $s_{j_{\ell}+1}$-element between the two $s_{j_{\ell}}$-elements;
\item[(3)] $H_{\{\rightarrow s_{j_{\ell}}\}}$ is alternating when the two $s_{j_{\ell}}$-elements are identified.
\item[(4)] $H$ is not a pseudo zigzag.
\end{itemize}
\smallskip

\item[\textbf{(RP)}] \textit{Right-Peaks.} $H\in$ (RP) if there exists $j_r\in \{2, \ldots, n\}$ such that:
\begin{itemize}
\item[(1)]  $H_{\{\rightarrow s_{j_r}\}}=P_{\rightarrow}(s_{j_r})$;
\item[(2)] There is no $s_{j_r-1}$-element between the two $s_{j_r}$-elements;
\item[(3)] $H_{\{\leftarrow s_{j_r}\}}$ is alternating when the two $s_{j_r}$-elements are identified.
\item[(4)] $H$ is not a pseudo zigzag.

\end{itemize}
\smallskip

\item[\textbf{(LRP)}] \textit{Left-Right-Peaks.} $H\in$ (LRP) if there exist $2\leq j_{\ell}< j_r\leq n$ such that:
\begin{itemize}
\item[(1)] LP(1), LP(2), RP(1) and RP(2) hold;
\item[(2)] $H_{\{s_{j_{\ell}},\ldots,s_{j_r}\}}$ is alternating when both the two $s_{j_{\ell}}$- and the two $s_{j_r}$-elements are identified.
\item[(3)] $H$ is not a pseudo zigzag.
\end{itemize}
\end{itemize}
\end{Definition}

\begin{figure}[hbtp]
\centering
\includegraphics[scale=0.75]{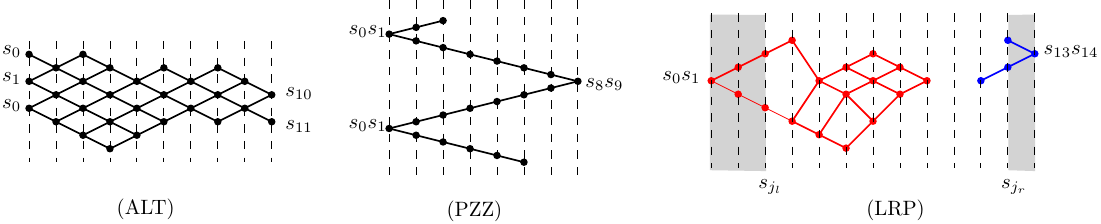}
\caption{Example of FC heaps in different families. The subheap of the (LRP) heap colored in red is in (LP), while the one in blue is in (RP).}
\label{heaps}
\end{figure}

The following result is proved in \cite[Theorem 3.10]{BJN_FC}.
\begin{Theorem}[Classification of FC heaps of type $\Dn$]
A heap of type $\Dn$ is fully commutative if and only if it belongs to one of the families of Definition \ref{family}.
\end{Theorem}

In the rest of the paper, we say that $w \in \FC(\Dn)$ is in $(\PZZ), (\ALT), (\LP), (\RP)$, or $(\LRP)$  if and only if the corresponding heap $H(w)$ is.

\begin{Remark}\label{rem:pzz-heaps}
With respect to \cite[\S 3.2]{BJN_FC}, we introduced a new family (PZZ) made of zigzags of the original classification, along with some additional heaps that previously belonged to different families. In particular, the heaps $H_{Z_1}:=H(s_0s_1s_2\cdots s_{n+1}s_{n+2})$ and $H_{Z_2}:=H(s_{n+2}s_{n+1}s_n\cdots s_1s_0)$ depicted in Figure \ref{heapspart} are now in (PZZ) while they
were alternating heaps in the previous classification.
Analogously, the left and right peaks having at least one initial or final endpoint situated in column 1 or $n+1$, and labeled by $s_0s_1$ or $s_{n+1}s_{n+2}$, respectively, are now in (PZZ); see Figure \ref{heapspart}, right.
The reasons for these choices will be clarified in Lemmas \ref{PZZheaps} and \ref{Pheaps}.
\end{Remark}

\begin{figure}[hbtp]
\centering
\includegraphics[scale=0.65]{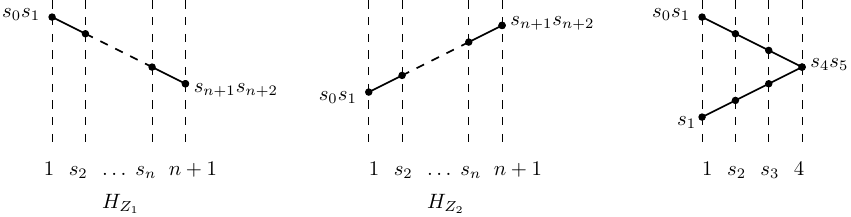}
\caption{}
\label{heapspart}
\end{figure}

%%%%%%%%%%%%%%%%%%%%%%%%%%%%%%%%%%%%%%%%%%%%%%%%%%%
\section{Decorated diagrams}\label{sec:undecorated}
%%%%%%%%%%%%%%%%%%%%%%%%%%%%%%%%%%%%%%%%%%%%%%%%%%%
In this section we consider a subset of the decorated diagrams defined in \cite{ErnstDiagramI} by Ernst to give his representation of $\TL(\widetilde{C}_{n+1})$. We then introduce a new set of relations to define a quotient subalgebra which will be the main object of our investigations.
\smallskip

A \textit{standard k-box} is a rectangle with $2k$ marked points called \textit{nodes} or \textit{vertices}, labeled as in Figure \ref{kbox}. We will refer to the top of the rectangle as the \textit{north face} and to the bottom as the \textit{south face}. 

Sometimes, it will be useful for us to think of the standard $k$-box as being embedded in the $xy$-plane. In this case, we put the lower left corner of the rectangle at the origin such that each node $i$ (respectively, $i'$) is located at the point $(i, 1)$ (respectively, $(i', 0)$).

\begin{figure}[hbtp]
\centering
\includegraphics[scale=0.7]{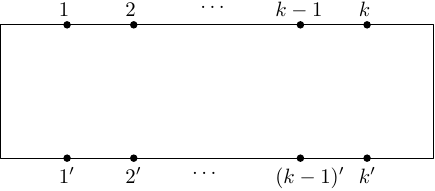}
\caption{The standard $k$-box.}
\label{kbox}
\end{figure}

A \textit{concrete pseudo k-diagram} consists of a finite number of disjoint plane curves, called \textit{edges}, embedded in the standard $k$-box with the following restrictions. The nodes are endpoints of edges. All other embedded edges must be closed (isotopic to circles) and disjoint from the box. We refer to a closed edge as a \textit{loop}. It follows that there cannot exist isolated nodes and that a single edge starts from each node (an example is given in Figure \ref{7box}). A \textit{non-propagating edge} is an edge that joins two nodes of the same face, while a \textit{propagating edge} is an edge that joins a node of the north face to a node of the south face. If an edge joins the nodes $i$ and $i'$ we call it a \textit{vertical edge}. We denote by $\ab(D)$ the number of non-propagating edges on the north face of a concrete pseudo $k$-diagram $D$.

\begin{figure}[hbtp]
\centering
\includegraphics[scale=0.45]{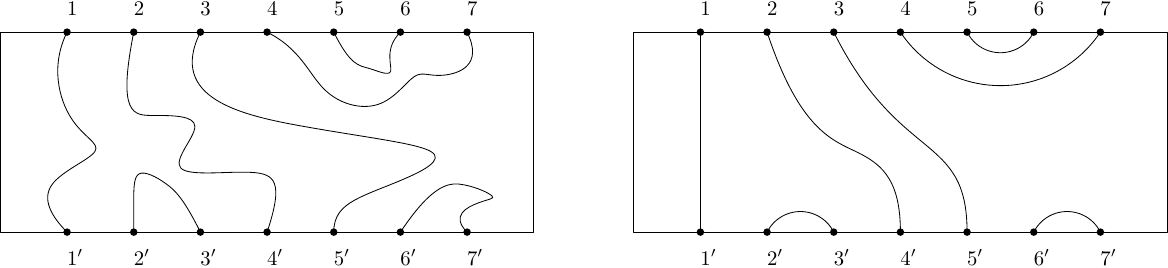}
\caption{A concrete pseudo 7-diagram and a standard representation.}
\label{7box}
\end{figure}

We define an equivalence relation on the set of the concrete pseudo $k$-diagrams. Two concrete pseudo $k$-diagrams are \textit{(isotopically) equivalent} if one can be obtained from the other by isotopically deforming the edges while preserving the faces of the rectangle setwise, in such a way that any intermediate diagram is also a concrete pseudo $k$-diagram.  
We define a \textit{pseudo k-diagram} as an equivalence class of concrete pseudo $k$-diagrams with respect to isotopy equivalence, and denote the set of the pseudo $k$-diagrams by $T_k(\emptyset)$.
Moreover, in our analysis it is important to fix a concrete representation in which no edge doubles back on itself, which we call \textit{standard}. What this means is that in such representation any vertical line intersects any edge, except vertical ones and loops, at most once. In Figure~\ref{7box} a concrete pseudo 7-diagram and one of its standard representations are depicted. 
\smallskip

Let $D$ be a concrete pseudo $k$-diagram. Consider the set $\Omega=\{\bullet, \circ\}$ and the monoid $\Omega^*$. Our goal is to adorn the edges of $D$ with elements of $\Omega$ which we call \textit{decorations}. In particular, $\bullet$ is called a \textit{L-decoration} and $\circ$ is called a \textit{R-decoration}. 
A finite sequence of decorations  $\mathbf{b}=x_1x_2\cdots x_p$ of $\Omega^*$ is called a \textit{block of decorations} of width $p$. We may adorn an edge with a finite sequence of blocks of decorations $\mathbf{b}_0,\ldots, \mathbf{b}_r$. If $\mathbf{b}$ is a block of decorations on an edge \textit{e}, we read off the sequence of decorations left to right if $e$ is a non-propagating edge, up to down if $e$ is a propagating edge. If $e$ is a loop edge, then the reading of the sequence of decorations depends on an arbitrary choice of the starting point and direction round the loop. We consider two sequences of blocks equivalent if one can be changed to the other or its opposite by any cyclic permutation.

We assign to each decoration $x_i$ a \textit{vertical position} equal to the $y$-coordinate in the $xy$-plane. 

\begin{Definition}
 Let $D$ be a concrete pseudo $k$-diagram decorated with at least one decoration in $\Omega$ and such that all decorations have different vertical positions. A $L$-\textit{strip} (respectively $R$-\textit{strip}) is a part of the diagram delimited by two horizontal lines that contains at least one $L$-decoration (respectively $R$-decoration) and no $R$-decorations (respectively $L$-decorations). A \textit{strip partition} of $D$ is a partition into \textit{strips} in which $L$-strips and $R$-strips alternate. 
\end{Definition}

Now we list the rules by which the edges of $D$ can be decorated. 
\begin{enumerate}
\item[(D0)] If $\textbf{a}(D)=0$ then the edges of $D$ are undecorated.
\end{enumerate}
Now suppose $\mathbf{a}(D)\neq 0$.
\begin{enumerate}
\item[(D1)] All decorated edges can be deformed so as to take $L$-decorations to the left wall of the standard box and $R$-decorations to the right wall simultaneously without crossing any other edge.

\item[(D2)] If \textit{e} is a non-propagating or a loop edge, then we allow adjacent blocks on \textit{e} to be conjoined to form a single block.
\item[(D3)] If $\mathbf{a}(D) > 1$ and \textit{e} is propagating, then as in (D2), we allow adjacent blocks on \textit{e} to be
conjoined to form a unique block.
\item[(D4)] If $\mathbf{a}(D)=1$, then we require the following.
 		\begin{enumerate}
            \item All decorations have different vertical positions. 
            
			\item All decorations occurring on propagating and loop edges must have vertical position lower (respectively, higher) than the vertical positions of decorations occurring on the (unique) non-propagating edge in the north face (respectively, south face) of \textit{D}.

            \item   All blocks in the same edge inside the same strip may be conjoined to form a single block.
   
			\end{enumerate}
\end{enumerate}

\begin{Remark}
   To better understand the conditions on the blocks of decorations listed in (D2) and (D4)(c), see Figure \ref{fig:strip}. In diagrams (A) and (B), since the \textbf{a}-value is greater than 1, then the two $\bullet$ on the edge $\{3,1'\}$ form a single block. In (C) and (D), the two $\bullet$ are in the same strip, so they belong to the same block, while in (E) they form two distinct blocks. 
\end{Remark}

\begin{figure}
    \centering
    \includegraphics[scale = 0.6]{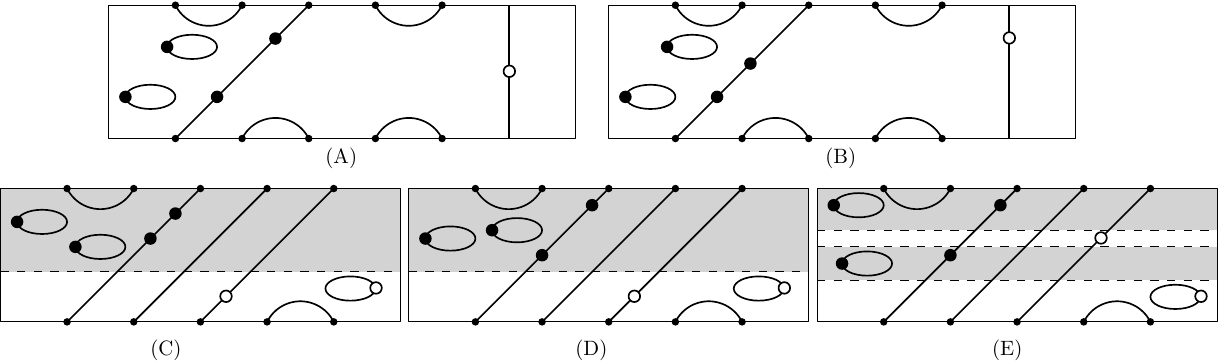}
    \caption{Diagrams (A) and (B) are $\Omega$-equivalent. Diagrams (C), (D) and (E) are isotopically equivalent but only (C) and (D) are $\Omega$-equivalent.}
    \label{fig:strip}
\end{figure}

We say that two concrete decorated pseudo \textit{k}-diagrams are $\Omega$-\textit{equivalent} if we can isotopically deform one diagram into the other in a way that any intermediate diagram is also a concrete pseudo \textit{k}-diagram decorated as above, and the relative vertical positions of the decorations in each block are preserved (i.e., the adjacency of the decorations is preserved). In addition, two diagrams with $\ab$-value equal to 1 are $\Omega$-\textit{equivalent} if furthermore, in any of their strip partitions, as well as in those of any intermediate diagram, the number of $L$ and $R$-strips and their relative positions are preserved; see Figure~\ref{fig:strip}.

\begin{Definition}
A $LR$-\textit{decorated pseudo k-diagram} is defined to be an equivalence class of $\Omega$-equivalent concrete decorated pseudo $k$-diagrams. We denote the set of $LR$-decorated pseudo $k$-diagrams by $\TLR$.
\end{Definition}

\begin{Definition}
We define $\PLR$ to be the free $\Zd$-module having the elements of $\TLR$ as a basis.
\end{Definition}

Now we define the product of two elements of $\TLR$ as the concatenation, and then extend this bilinearly to define a multiplication in $\PLR$. To concatenate two diagrams $D,D'\in \TLR$, we place $D'$ on top of $D$ so that node $i$ of $D$ coincides with node $i'$ of $D'$, conjoin adjacent blocks maintaining $\Omega$-equivalence and then rescale vertically by a factor of $1/2$. 

\begin{figure}[hbtp]
\centering
\includegraphics[scale=0.7]{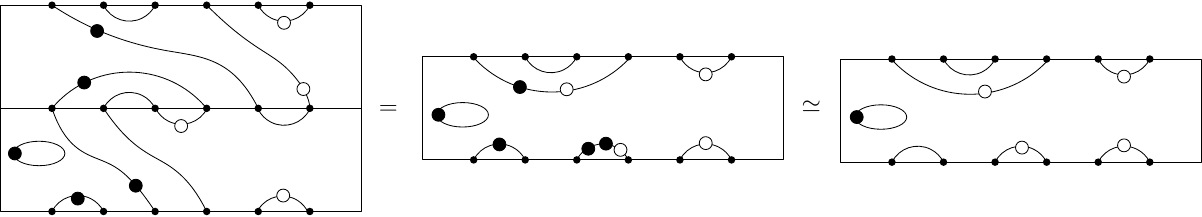}
\caption{Product in $\mathcal{P}^{LR}_6$ and reduction in $\widehat{\mathcal{P}}^{LR}_6$.}
\label{prodec}
\end{figure}

\begin{Proposition}\label{decoralgebra}
The module $\PLR$ with the aforementioned product is a well-defined associative $\Zd$-algebra, with identity element the pseudo $k$-diagram $I_k$ made of $k$ undecorated vertical edges. 
\end{Proposition}

\begin{proof}
By \cite[Proposition 3.3.5]{G1}, the multiplication by concatenation defined above preserves the condition (D1).
 Let $D$ and $D'$ be two diagrams with $\mathbf{a}(D)=\mathbf{a}(D')=1$. If $\mathbf{a}(D'D)>1$, then there are no concerns, since all the blocks on each edge can be conjoined. Else, if $\mathbf{a}(D'D)=1$, then (D4) holds. Indeed, if the unique non-propagating edge $e'$ in the south face of $D'$ is $\{i', (i+1)'\}$, then the unique non-propagating edge $e$ in the north face of $D$ is either (a) $\{i-1, i\}$, or (b) $\{i, i+1\}$, or (c) $\{i+1, i+2\}$ (for more details, see \cite[Proposition 3.2.7]{ErnstDiagramI}). If the lowest strip of $D'$ and the highest strip of $D$ are associated to the same decoration, then we conjoin the two strips and the blocks on the edges inside those strips (if any). Otherwise, $D'D$ contains, from the top, the strips of $D'$ followed by those of $D$ and there is no conjoining of blocks. In both cases, (D4) holds; see Figure \ref{diagrel}. By construction, the multiplication is associative (see also \cite[Proposition 3.1.2]{G1}) and this concludes the proof.
\end{proof}

Note that $\PLR$ is an infinite dimensional algebra since, for instance, when $D$ has a single propagating edge, there is no limit to the number of decorations on such edge, or when $D$ has no propagating edges, there is no limit to the number of loops with both decorations. 

We introduce the notation for three types of loop edges in Figure \ref{loop-D}.

 \begin{figure}[h]
	\centering
	\includegraphics[width=0.45\linewidth]{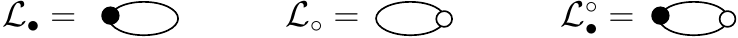}
	\caption{Special types of loops.}
	\label{loop-D}
 \end{figure}

\begin{figure}[hbtp]
\centering
\includegraphics[scale=0.6]{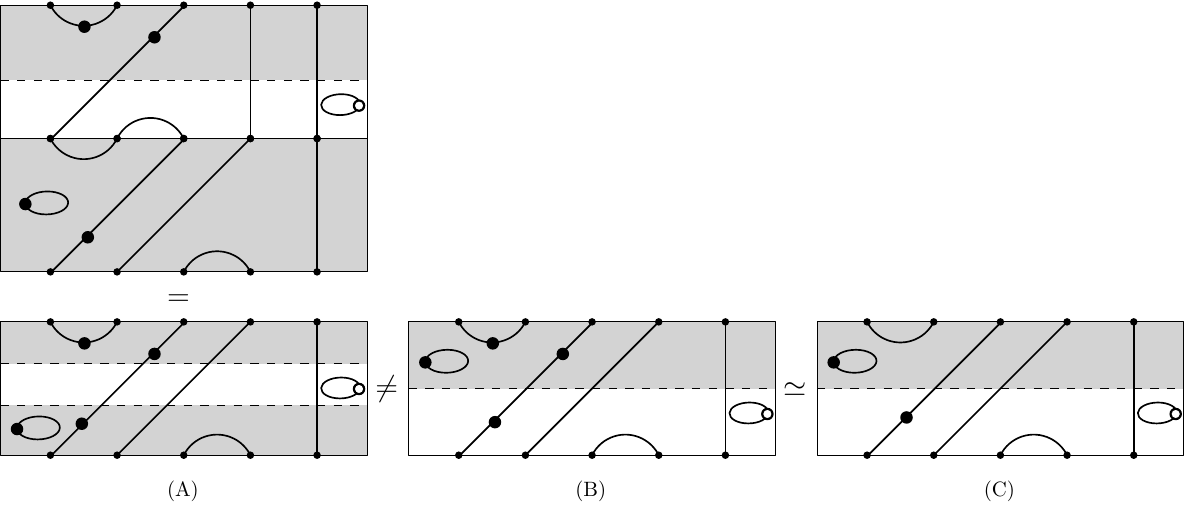}
\caption{Product in $\mathcal{P}^{LR}_5$ with $\ab$-value 1. Diagrams (A) and (C) are irreducible while (B) is not; (B) and (C) are equal in $\widehat{\mathcal{P}}^{LR}_5(\Omega)$.}
\label{diagrel}
\end{figure}

From now on we follow the setup of Bergman's Diamond Lemma \cite [\S 1]{Ber}. We define a \textit{reduction system} as depicted in Figure \ref{rel}, where the lines represent a portion of an edge (loops included). 

\begin{figure}[hbtp]
\centering
\includegraphics[scale=0.7]{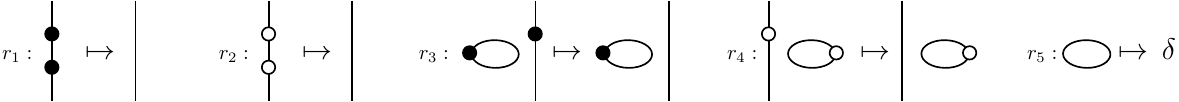}
\caption{The reduction system in $\PLR$.}
\label{rel}
\end{figure}

More precisely the reductions $r_i$ described in Figure \ref{rel} are defined as follows:
\begin{enumerate}   
    \item $r_1$ and $r_2$ reduce the number of decorations and they are applied to adjacent decorations in the same block;
     \item if $\lob$ (respectively, $\loc$) occurs in $D$, then $r_3$ (respectively, $r_4$) removes a $\bullet$ (respectively, $\circ$) on another edge (loop edge included) with the only restriction that if $\ab(D)=1$ then $\lob$ (respectively, $\loc$) and $\bullet$ (respectively, $\circ$)  must belong to the same strip. 
     \item $r_5$ removes an undecorated loop edge multiplying $D$ by a factor $\delta$.
\end{enumerate} 
Clearly, if $\ab(D)>1$, a finite sequence of applications of $r_3$ (respectively, $r_4$) removes all $\bullet$ (respectively, $\circ$) in $D$ except the one on the last remaining loop, while in the case $\ab(D)=1$, it only removes all $\bullet$ (respectively, $\circ$) that belong to the same $L$-strip (respectively, $R$-strip).
\smallskip

We say that a $LR$-decorated diagram is \textit{irreducible} if no relation can be applied. Note that each block of decorations of an irreducible diagram cannot contain two adjacent decorations of the same type.

Define the function $sh: \TLR \rightarrow T_k(\emptyset)$ such that $sh(D)$ is the \textit{shape} of  $D$, i.e. $sh(D)$ is an element of $T_k(\emptyset)$ which is equal to $D$ with all decorations and loops removed. Also define a function $h: \TLR \rightarrow \mathbb{N}$ where $h(D)$ is the sum of the number of decorations and the number of loops of $D$.
\medskip

Define $\leq_{\widehat{\mathcal{P}}}$ on $\TLR$ via 
$$D<_{\widehat{\mathcal{P}}} D'\mbox{ if and only if }sh(D)=sh(D')\mbox{ and }h(D)<h(D').$$ 
Let $D,D',A,B$ be elements of $\TLR$ with $D<_{\widehat{\mathcal{P}}}D'$. We can easily verify that $sh(ADB)=sh(AD'B)$ and $h(ADB)<h(AD'B)$, so that $ADB<_{\widehat{\mathcal{P}}}AD'B$. Therefore, $\leq_{\widehat{\mathcal{P}}}$ is a semigroup partial order on $\TLR$. Furthermore, we note that the reductions preserve the shape of the edges and decrease the number of decorations. Hence $\leq_{\widehat{\mathcal{P}}}$ is compatible with the set of reductions. Moreover, $\leq_{\widehat{\mathcal{P}}}$ clearly satisfies the descending chain condition, that is every sequence of reductions ends in a finite number of steps.\\

Now we want to show that all of the potential ambiguities are resolvable. Note that $r_5$ commutes with any other reduction. Moreover, if there are two reductions, one involving only $ \bullet$-decorations and one involving only $\circ$-decorations which can be applied at the same time, then they commute, so the corresponding ambiguity is resolvable.

 So we only need to check ambiguities involving the same type of decorations, but we know that all such ambiguities are resolvable since they coincide with those appearing in the Coxeter type finite $D$ \cite[Theorem 4.2]{Green-general}.

Let $\PLRk$ be the quotient of $\PLR$ modulo the two-sided ideal generated by the relations determined by the reductions in Figure \ref{rel}. Since all the ambiguities are resolvable, by \cite[Theorem 1.2]{Ber} we have the following result.

\begin{Proposition}
 The set of $LR$-decorated irreducible diagrams forms a basis for $\PLRk$.
\end{Proposition}

Now define the \textit{simple diagrams} $D_0, \ldots, D_{n+2}$ as in Figure \ref{simple}. Since they are irreducible, they are basis elements of $\PLRn$.  

\begin{figure}[hbtp]
\centering
\includegraphics[scale=0.6]{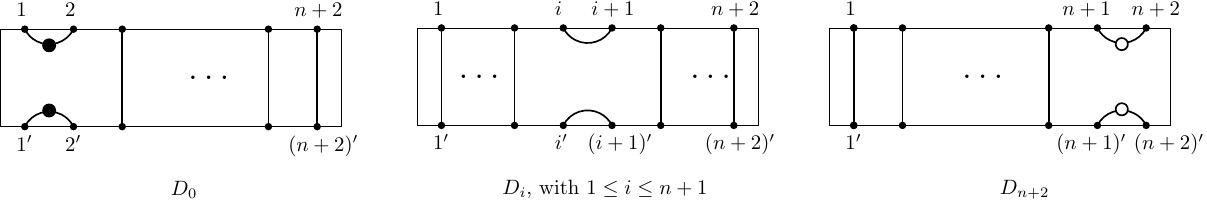}
\caption{The simple diagrams.}
\label{simple}
\end{figure}

It is easy to prove that the simple diagrams satisfy the relations (d1)-(d3)  listed in Definition~\ref{def:tl-algebras}, where $b_i$'s are replaced by $D_i$'s.

\begin{Definition}
Let $\DD$ be the $\Zd$-subalgebra of $\PLRn$ generated as a unital algebra by the simple diagrams with multiplication inherited by $\PLRn$.
\end{Definition}

In the next section we give an explicit description of the diagrams belonging to this subalgebra. 

%%%%%%%%%%%%%%%%%%%%%%%%%%%%%%%%%%%%%%%%
\section{Admissible diagrams of type $\widetilde{D}$}\label{sec:admissible}
%%%%%%%%%%%%%%%%%%%%%%%%%%%%%%%%%%%%%%%%

In this section we introduce a family of diagrams, called \textit{admissible} and we consider the $\Zd$-module they generate inside $\PLRn$. The goal of the next sections will be to show that such a module and $\DD$ coincide as subalgebras of $\PLRn$.

\begin{Definition}\label{def:admissible}

An irreducible diagram $D\in T_{n+2}^{LR}(\Omega)$ is called a $\widetilde{D}$-\textit{admissible} diagram if it satisfies the following conditions:

\begin{itemize}

\item[(A1)] The only possible loops in $D$ are those depicted in Figure~\ref{loop-D}. If a $\locb$ occurs, the other two types of loops cannot appear. Moreover, if $\ab(D)>1$ then there is at most one occurrence of $\lob$ and at most one occurrence of $\loc$.
\end{itemize}

Assume $\ab(D)>1$.
\begin{itemize}
\item[(A2)] Both the sum of the number of $\locb$ and $\bullet$ on non-loop edges, and the sum of the number of $\locb$ and $\circ$ on non-loop edges, must be even.
\smallskip
\item[(A3)] If \textit{e} is the vertical edge $\{1,1'\}$ (respectively $\{n+2, (n+2)'\}$), then exactly one of the following three conditions is met:
\begin{itemize}
\item[(a)] \textit{e} is undecorated;
\item[(b)] \textit{e} is decorated by a single $\circ$ (respectively $\bullet$) and no $\loc$ (respectively $\lob$) occurs;
\item[(c)] \textit{e} is decorated by an alternating sequence of $\bullet$ and $\circ$ and there are no loops.
\end{itemize}
\smallskip
\end{itemize}

Assume $\ab(D)=1$.

\begin{itemize}
\item[(A4)] The western side of $D$ is equal to one of those depicted in Figure  \ref{west}, where the number of $\lob$ might be 0 and $\du, \dd \in \{\emptyset, \bullet\}$. 
In each $\bullet$-strip, there is a unique $\lob$ and no other decorations, except when $\du$ or $\dd$ of Figures \ref{west} (B)-(C)-(D) are not empty. In these cases, the highest (respectively, lowest) strip of the diagram contains two $\bullet$ and no $\lob$. %Moreover, the diagram in Figure \ref{west} (B) with no $\lob$ can only occur if either $\du$ or $\dd$ is empty and if $D$ is not decorated with any $\circ$.
We have similar restrictions for the eastern side of $D$, where the $\bullet$ are replaced by the $\circ$ and the $\lob$ with the $\loc$ (see Remark \ref{rem:flip} for an analogous description).

\begin{figure}[hbtp]
	\centering		
		\includegraphics[height=27mm, keepaspectratio]{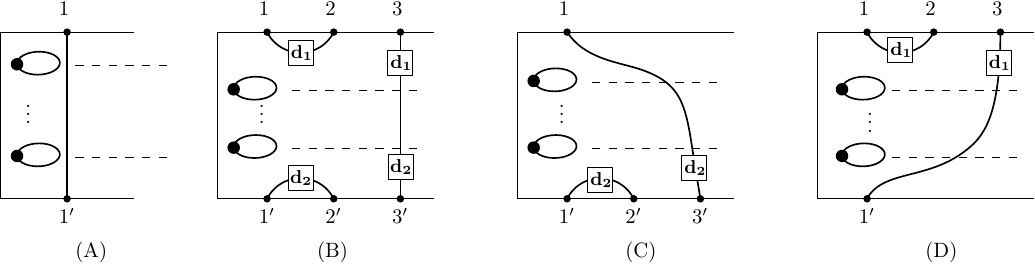}
		\caption{Western side of an admissible diagram with $\ab(D)=1$. The number of $\lob$ in each diagram can be zero.}
	\label{west}
\end{figure}

\end{itemize}

\end{Definition}

\begin{Remark}\label{alternatingdec} \
\begin{itemize}
    \item[(a)] An edge $e$ of an admissible diagram $D$ can be decorated by both $\bullet$ and $\circ$ only if $D$ has no propagating edges or if $e$ is the unique propagating edge of $D$. 
    \item[(b)] If $\ab(D)=1$, then the $\bullet$-strips and the $\circ$-strips must alternate. 
    \item[(c)] In a standard representation, all the loops $\lob$ (respectively $\loc$) will be depicted vertically aligned on the left (respectively right) of the leftmost (respectively rightmost) propagating edge of the diagram (Figure~\ref{west}). When $D$ has no propagating edges, the $\locb$ will be drawn horizontally aligned in the middle of the picture (Figure~\ref{ALTexample1}).
    \item[(d)] The diagram in Figure \ref{diagrel}(C) is an example of an irreducible diagram that is not admissible, since it does not verify (A2).
\end{itemize}

\end{Remark} 

\begin{Definition}
We denote by $\dbD$  the set of all $\widetilde{D}$-admissible $(n+2)$-diagrams and by $\mathcal{M}[\dbD]$ the $\Zd$-submodule of $\PLRn$ they span.
\end{Definition}

Since the admissible diagrams $\dbD$ are irreducible, they are independent.
%The diagrams in $\dbD$ are independent since they are irreducible, hence we have the following result.

\begin{Proposition}\label{prop:basis}
    The set of admissible diagrams $\dbD$ is a basis for the module $\mathcal{M}[\dbD]$.
\end{Proposition}

\smallskip

We now classify the admissible diagrams in three main families.

\begin{Notation}\label{not:l-r}
We denote by $\mathtt{l}$ \textit{the total number of $\lob$} and by $\mathtt{r}$ \textit{the total number of $\loc$} occurring in an admissible diagram. 
\end{Notation}

\begin{Definition}[PZZ-diagrams]\label{defPZZ}
Let $D\in \dbD$. We say that $D$ is a \textit{PZZ-diagram} if $\ab(D)=1$ and $\mathtt{l},\mathtt{r}\geq 1$. 
In particular, we say that $D$ is of \textit{type} $\langle a, b \rangle$ where $a,b \in \{\bullet, \circ\}$ if the highest loop in $D$ is $\mathcal{L}_a$ and the lowest loop in $D$ is $\mathcal{L}_b$.
\end{Definition}

By (A4) and Remark \ref{alternatingdec}(b) in a PZZ-diagram we have that $|\mathtt{l}-\mathtt{r}|\leq 1$.

\begin{figure}[hbtp]
\centering
\includegraphics[height=19mm, keepaspectratio]{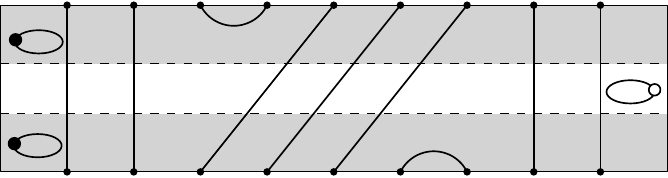}
\caption{A PZZ-diagram of type $\langle\bullet \, \bullet\rangle$ with $\mathtt{l}=2$ and $\mathtt{r}=1$.}
\label{PZZdiag}
\end{figure}

The two specific PZZ-diagrams depicted in Figure \ref{degeneratePZZ} will be used later. Note that $Z_1$ is of type $\langle \bullet \, \circ \rangle$, $Z_2$ is of type $\langle \circ \, \bullet \rangle$ and they both have $\mathtt{l}=\mathtt{r}=1$.

\begin{figure}[hbtp]
\centering
\includegraphics[height=22mm, keepaspectratio]{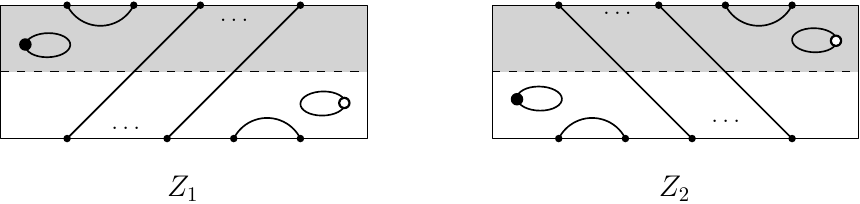}
\caption{Special PZZ-diagrams.}
\label{degeneratePZZ}
\end{figure}

\begin{Definition}[P-diagrams]\label{defP}
Let $D\in \dbD$ that is not a PZZ-diagram. We say that:
\begin{itemize}
\item[(LP)] $D$ is of \textit{type} LP if it has a single $\lob$ and there exists $j_\ell\in \{2,\ldots, n+1\}$ such that the leftmost $j_{\ell}-1>0$ edges  are vertical and undecorated, and the one leaving node $j_{\ell}$ is not an undecorated vertical edge.

\item[(RP)] $D$ is of \textit{type} RP if it has a single $\loc$ and there exists $j_r\in \{2,\ldots, n+1\}$ the rightmost $n+2-j_r>0$ edges are vertical and undecorated, and the one leaving node $j_r$ is not an undecorated vertical edge.%; equivalently $D$ is the conjugate of a diagram of type LP.
\item[(LRP)] $D$ is of \textit{type} LRP if it is of both types LP, RP and if $\ab(D)>1$. 
The condition $\ab(D)>1$ implies that $j_r - j_{\ell}>2$.
\end{itemize}
If $D$ is of type LP, RP or LRP then we say that $D$ is a \textit{P-diagram}.
\end{Definition}

\begin{figure}[hbtp]
\centering
\includegraphics[height=22mm, keepaspectratio]{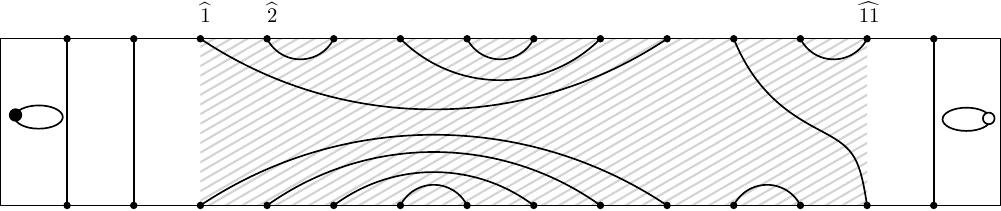}
\caption{A P-diagram of type LRP, with $j_{\ell}=3$,  $j_r=13$ and $n=12$. The shaded part denotes its inner alternating diagram, as defined in Definition \ref{def:altportion}.}
\label{LRPdiag}
\end{figure}

\begin{Definition}[ALT-diagrams]\label{defALT}
Let $D\in \dbD$. If $D$ is not a P-diagram or a PZZ-diagram, then we say that $D$ is an \textit{ALT-diagram}. In particular, the identity diagram is an ALT-diagram.
\end{Definition}

\begin{figure}[hbtp]
\centering
\includegraphics[height=17mm, keepaspectratio]{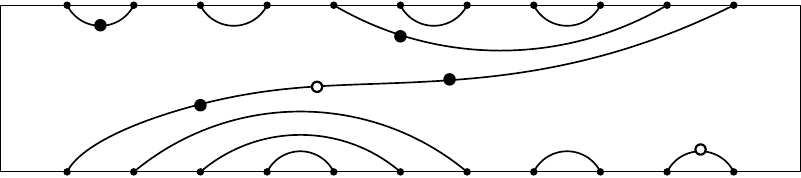}
\caption{An ALT-diagram with $n=9$.}
\label{ALTdiag}
\end{figure}

The conditions required for the ALT, P and PZZ diagrams are consistent with (A1)--(A4); moreover, the three families form a partition of the set of admissible diagrams.

\begin{Definition}\label{def:altportion}
Given a P-diagram $D$, we define the \textit{inner alternating diagram}, denoted by $\widehat{D}$, as the diagram embedded in the standard $(j_r-j_{\ell}+1)$-box and delimited by the vertical edges $\{j_{\ell}-1, (j_{\ell}-1)'\}$ and $\{j_r+1, (j_r+1)'\}$ of $D$. We relabel the nodes of $\widehat{D}$ by $\hat{1}, \hat{2}, \hat{3}, \ldots$, where $\hat{k}=k+j_{\ell}-1$. If $D$ is of type LP (respectively, RP) we set $j_r=n+2$ (respectively, $j_{\ell}=1$). 
\end{Definition}
Note that $\widehat{D}$ is an ALT-diagram. For example, in Figure~\ref{LRPdiag}, $\widehat{D}$ is represented by the shaded part.
\smallskip

%Sometimes the $\bullet$ and $\circ$ decorations present a kind of symmetry. 
For our aims, it is useful to introduce the following definition. The notation $a^*$ means that $a^*$ might be either a node $a$ in the north face or a node $a'$ in the south face of $D$.

\begin{Definition}\label{def:flipped}
Let $D\in \dbD$. The \textit{conjugate diagram} $F(D)$ is the diagram having the set of edges equal to $\{F(e) \mid e \mbox{ is an edge of } D\}$, where $F(e)$ is defined as follows.\\
If $e$ is a loop, then 
\begin{itemize}
    \item[(f1)] $F(\lob)=\loc$, $F(\loc)=\lob$ and $F(\locb)=\locb$, with the same vertical positions.
\end{itemize}
Otherwise,
\begin{itemize}
    \item[(f2)] if $e=\{i^*,j^*\}$ is a non-propagating edge, then $F(e)=\{(n+3-j)^*, (n+3-i)^*\}$;
    \item[(f3)] if $e=\{i, j'\}$ is a propagating edge, then $F(e)=\{n+3-i, (n+3-j)'\}$;
    \item[(f4)] if $e$ is a non-propagating edge decorated with $\bullet \, \circ$, then $F(e)$ is decorated with $\bullet \, \circ$. Otherwise, if $e$ has a $\bullet$ (respectively $\circ$) then $F(e)$ has a $\circ$ (respectively $\bullet$) in the same vertical position.        
\end{itemize}
\end{Definition}

\begin{Remark}\label{rem:flip} 
One can note that this operation is left-right reversal composed with exchanging the two types of decoration. In particular, $F(D)\in \dbD$, $F(F(D))=D$ and $F(DD')=F(D)F(D')$ for $D, D'\in \dbD$. 
\end{Remark}
For example $F(Z_1)=Z_2$ and $F(Z_2)=Z_1$, see Figure~\ref{degeneratePZZ}. Moreover, in Definition~\ref{def:admissible}(A4) the eastern side of $D$ can be obtained by applying $F$ to the diagrams depicted in Figure~\ref{west}. Note that $F(D)$ is an ALT (respectively, P, PZZ) diagram if and only if $D$ is an ALT (respectively, P, PZZ) diagram. In particular, a RP-diagram is the conjugate of a LP-diagram, and vice versa, while the conjugate of a LRP-diagram is a LRP-diagram.

\subsection{Length of an admissible diagram.}\label{subsec:length}
In this section we introduce a length function that will be a fundamental tool in the proof of the injectivity of our representation. In the diagram algebras of types $A$ and $\widetilde{A}$ the length of a diagram is defined by counting the number of intersections between all the vertical lines $i+1/2$ in the standard box and the edges of the diagram (see \cite[Lemma 3.3]{Green-general}, \cite[Definition 4.4.2]{FanGreen_Affine}). To define an analogous length function in type $\widetilde{D}$, for instance in the case of an ALT-diagram, one may count the intersections between such vertical lines and the edges of an isotopically equivalent diagram obtained by stretching the decorated edges to touch the left or the right wall of the box without crossing each other, as sketched in Figure~\ref{stretch}. For the P and PZZ-diagrams the geometric interpretation is more subtle. 
\begin{figure}[hbtp]
\centering
\includegraphics[height=21mm, keepaspectratio]{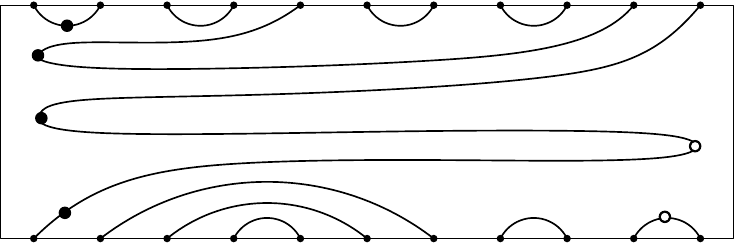}
\caption{Stretched representation of the diagram in Figure~\ref{ALTdiag}. The length of this diagram is 29.}
\label{stretch}
\end{figure}
The algebraic formulation of the length function that arises from this approach is the following. 

From now on, we consider the standard representation of a diagram (see Section~\ref{sec:undecorated}) and we assign to any edge $e$ of $D$ a \textit{weight} $\weight(e)$ whose value depends only on the decorations $\bullet$ and $\circ$. We set $\weight(\lob)=\weight(\loc)=1$ and $\weight(\locb)=n+1$. If $e$ is undecorated, then $\weight(e)=0$. For $k\in\mathbb{N}_0$,  set $(\circ\,\bullet)^k:=\underbrace{\circ\bullet \cdots \circ\bullet}_{2k}$ and $(\bullet\,\circ)^k:=\underbrace{\bullet\circ \cdots \bullet\circ}_{2k}$.

If $i< j$ and $e=\{i,j\}$, $e=\{i', j'\}$ or $e=\{i,j'\}$, then we set

\begin{align}\label{weight1}
\weight(e):=\begin{cases}
i-1 +k(n+1), & \mbox{ if }e\mbox{ is decorated with} \ \bullet\mbox{ followed by } (\circ\,\bullet)^k;  \\
n+2-j + k(n+1), & \mbox{ if }e\mbox{ is decorated with} \ \circ\mbox{ followed by }(\bullet\,\circ)^k;  \\
(k+1)(n+1)-j+i, & \mbox{ if }e\mbox{ is decorated with} \ (\bullet\,\circ)^{k+1}; \\
(k+1)(n+1),& \mbox{ if }e\mbox { is decorated with} \ (\circ\,\bullet)^{k+1}.
\end{cases}
\end{align}

If $i\geq j$ and $e=\{i, j'\}$, 
\begin{align}\label{weight2}
\weight(e):=\begin{cases}
j-1 +k(n+1), & \mbox{ if }e\mbox{ is decorated with} \ \bullet\mbox{ followed by } (\circ\,\bullet)^k; \\
n+2-i + k(n+1), & \mbox{ if }e\mbox{ is decorated with} \ \circ\mbox{ followed by }(\bullet\,\circ)^k; \\
(k+1)(n+1), & \mbox{ if }e\mbox{ is decorated with} \ (\bullet\,\circ)^{k+1}; \\
(k+1)(n+1)-i+j, & \mbox{ if }e\mbox{ is decorated with} \ (\circ\,\bullet)^{k+1}.
\end{cases}
\end{align}

\begin{Remark}\label{rem:endpoints}
    Note that, by definition, the weight of an edge decorated with only a $\bullet$ (respectively, $\circ$) depends only on its leftmost (respectively, rightmost) endpoint. In particular, if the edge $e=\{1^*,j^*\}$ (respectively, $e=\{i^*, (n+2)^*\}$) is decorated only with a $\bullet$ (respectively, $\circ$), then $\weight(e)=0$.
\end{Remark}

\begin{figure}[hbtp]
\centering
\includegraphics[scale=0.7]{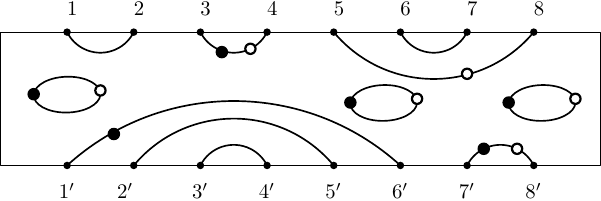}
\caption{The weights of the non-loop edges are all 0 except $\weight(3,4)=6$ and $\weight(7',8')=6$. Moreover, $\weight(\locb)=7$ and $\ell(D)=41$.}
\label{ALTexample1}
\end{figure}

\begin{Notation}\label{not:pesi}
We set $\nu_i(D)$ \textit{the number of non-loop edges of $D$ that intersect the vertical lines} $i+1/2$, with $i\in\lbrace1,\ldots,n+1\rbrace$. Note that $\nu_i(D)$ is well-defined since we work with a standard representation of $D$ and it is always an even number. Also define 
$$\nu(D):=\sum_{i=1}^{n+1}\nu_i(D) \qquad \mbox{ and } \qquad \weight(D):=\sum \weight(e),$$ 
where the second sum runs over all the edges of $D$, loops included.
\end{Notation}

\begin{Definition}[Length of a $\widetilde{D}$-admissible diagram] \label{deflength}
We define the \textit{length} $\ell(D)$ of an admissible diagram $D \in \dbD$ as follows.
\begin{itemize}
\item[(a)] If $D$ is an ALT-diagram, we define 
\begin{equation}\label{eq:length}
    \ell(D):= \frac{1}{2}\nu(D) + \weight(D).
\end{equation}
In particular, if $D$ is the identity diagram then $\ell(D)=0$.
\item[(b)] If $D$ is a PZZ-diagram, we set 
\begin{align*}
\ell(D):=\begin{cases} 3-i+n-j + (\mathtt{l}+\mathtt{r})(n+1), &\mbox{ if $D$ is of type $\langle \circ \, \circ \rangle$};\\
1-i+j+(\mathtt{l}+\mathtt{r})(n+1), &\mbox{ if $D$ is of type $\langle \circ \, \bullet \rangle$};\\
1+i-j+(\mathtt{l}+\mathtt{r})(n+1), &\mbox{ if $D$ is of type $\langle \bullet \, \circ \rangle$};\\
i-n + j-1+(\mathtt{l}+\mathtt{r})(n+1), &\mbox{ if $D$ is of type $\langle \bullet \, \bullet \rangle$},
\end{cases}
\end{align*}
where $\{i, i+1\}$ is the unique non-propagating edge on the north face and $\{j', (j+1)'\}$ is the unique non-propagating edge on the south face of $D$, while $\mathtt{l}$, $\mathtt{r}$, and $\langle \_  \, \_ \rangle$ are as in Notation~\ref{not:l-r} and Definition~\ref{defPZZ}.

\item[(c)] If $D$ is a P-diagram, we set
\begin{align*}
\ell(D):=\begin{cases} 2j_\ell-1 + \ell(\widehat{D}), &\mbox{ if }D\mbox{ is of type LP};\\
2n-2j_r +5 + \ell(\widehat{D}), &\mbox{ if }D\mbox{ is of type RP};\\
2n + 2j_\ell-2j_r + 4 +\ell(\widehat{D}), &\mbox{ if }D\mbox{ is of type LRP},\\
\end{cases}
\end{align*}
where $j_\ell$ and $j_r$ are as in Definition~\ref{defP}, $\widehat{D}$ is the inner alternating diagram, and $\ell(\widehat{D})$ is defined in (\ref{eq:length}).
\end{itemize}
\end{Definition}

\noindent Note that $\ell(D)=0$ if and only if $D$ is the identity diagram. In Figures~\ref{ALTexample1} and \ref{PZZexample1} the length is computed in some examples.

\begin{figure}[hbtp]
\centering
\includegraphics[height=16mm, keepaspectratio]{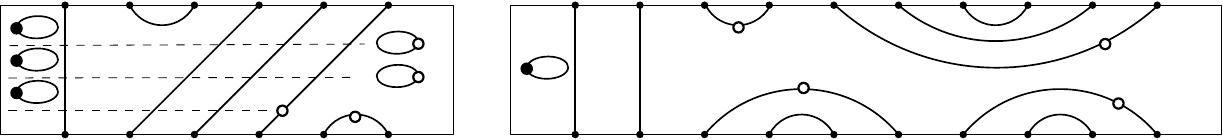}
\caption{A PZZ-diagram of type $\langle\bullet \,  \,\bullet\rangle$  on the left, with $\mathtt{l}=3$, $\mathtt{r}=2$ and $\ell(D)=27$ and a LP-diagram, right, with $j_\ell=3$ and $\ell(D)=24$.}
\label{PZZexample1}
\end{figure}

\begin{Lemma}\label{rem:length-conj}
    Let $D\in \dbD$, then $\ell(F(D))=\ell(D)$.
\end{Lemma}
    
  \begin{proof} 
    
    By Definition \ref{def:flipped}, $$F(e)=\begin{cases}
    \{n+3-j, n+3-i\}, \text{ if $e=\{i^*,j^*\}$ is non-propagating;}\\
\{n+3-i, (n+3-j)'\}, \text{ if $e=\{i,j'\}$ is propagating.}
    \end{cases}$$ 
    Moreover, $F(e)$ is decorated following (f4),
    so, by Equations \eqref{weight1} and \eqref{weight2}, $\weight(F(e))=\weight(e)$. Without considering the decorations, the shape of $F(D)$ is obtained by the shape of $D$ by reflecting through the central vertical axis of $D$. Hence, the number of intersections with the vertical lines $i+1/2$ is preserved. These two facts imply that if $D$ is an ALT-diagram, $\ell(F(D))=\ell(D)$.

    If $D$ is a PZZ-diagram, by Definition \ref{def:flipped}, the conjugate of a PZZ-diagram of type $\langle \bullet,  \bullet\rangle$ is a PZZ-diagram of type $\langle \circ,  \circ\rangle$, while the conjugate of a PZZ-diagram of type $\langle \bullet,  \circ\rangle$ is of type $\langle \circ,  \bullet\rangle$. Hence, when applying the Definition \ref{deflength}(b), observing that $F$ preserves the quantity $\mathtt{l}+\mathtt{r}$, it follows that $\ell(F(D))=\ell(D)$. 
    
    Finally, if $D$ is a P-diagram, since $F$ exchanges the values $j_\ell$ and $j_r$, and $\widehat{D}$ is an ALT-diagram, then by a quick computation, $\ell(F(D))=\ell(D)$.
\end{proof}

%%%%%%%%%%%%%%%%%%%%%%%%%%%%%%%%%%%%%%%%%%%%%%%
\section{The simple edges and the cut and paste operation}\label{sec:simple edges-cp}
%%%%%%%%%%%%%%%%%%%%%%%%%%%%%%%%%%%%%%%%%%%%%%%

\begin{Definition}\label{def:simpleedge}
A non-propagating edge on the north face $e=\{i, i+1\}$ of $D\in \dbD$ is called a \textit{simple edge of the form}:
\begin{enumerate}
\item[$(\smile)$] if $e$ is undecorated;
\item[$(\stackrel{\bullet}{\smile})$] if $e=\{1,2\}$ and it is decorated by a single $\bullet$;
\item[$(\stackrel{\circ}{\smile})$] if $e=\{n+1, n+2\}$ and it is decorated by a single $\circ$.
\end{enumerate}
\end{Definition}

\begin{Definition}\label{deftypes}
Let $e=\{i,i+1\}$ be a simple edge  in $D\in \dbD$ of any of the three forms above. We say that $e$ is of \textit{type}:

\begin{itemize}
\item[($U$)] $U_R$ (respectively, $U_L$) if $D$ contains the edge $f=\{i+2, k\}$ (respectively $\{j,i-1\}$) decorated with at least one $\bullet$ (respectively $\circ$);
\smallskip
\item [($N$)] if $D$ contains an edge $f=\{j,k\}$ with $j<i<i+1<k$;
\smallskip
\item [($P$)] $P_R$ (respectively $P_L$) if $D$ contains the edge $f=\{i+2, j'\}$ (respectively $\{i-1, k'\}$) with $j\leq i$ (respectively $k\geq i+1$);
\smallskip
\item[($S$)] $S_R$ (respectively $S_L$) if $D$ contains the edge $f=\{i+2, (i+2)'\}$ (respectively $\{i-1, (i-1)'\}$) decorated with at least one $\bullet$ (respectively $\circ$).
\end{itemize}
Moreover we say that the simple edge $e=\{1,2\}$ (respectively $e=\{n+1, n+2\}$) is of \textit{type}:
\begin{itemize}
\item[($P^*$)] $P_R^\bullet$ (respectively $P_L^\circ$) if $e$ is of type $P_R$ (respectively, $P_L$), $D$ contains a $\lob$ (respectively a $\loc$) and the edge $\{3, 1'\}$ (respectively $\{n, (n+2)'\}$) is undecorated.
\end{itemize}
The various situations are illustrated in Figure~\ref{types}.
\end{Definition}

\begin{figure}[hbtp]
\centering
\includegraphics[scale=0.7]{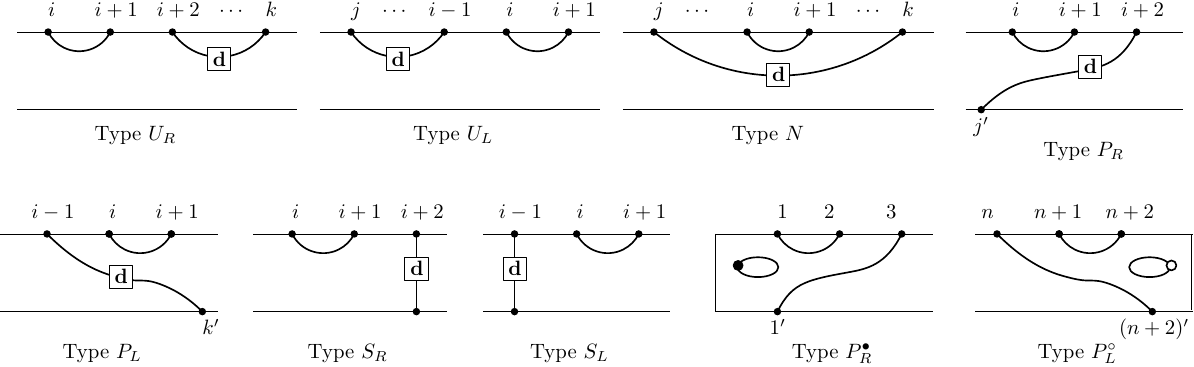}
\caption{Types of simple edges and their neighbors. In cases $N$ and $P$, the block $\bf d$ can be empty.}
\label{types}
\end{figure}

\begin{Definition}\label{def:basic-diag}
An ALT-diagram $D$ is a \textit{top basic diagram} if it consists of simple edges and a finite number (possibly zero) of undecorated vertical edges and loop edges (see Figure~\ref{basic}). 
\end{Definition}

\begin{figure}[hbtp]
\centering
\includegraphics[scale=0.6]{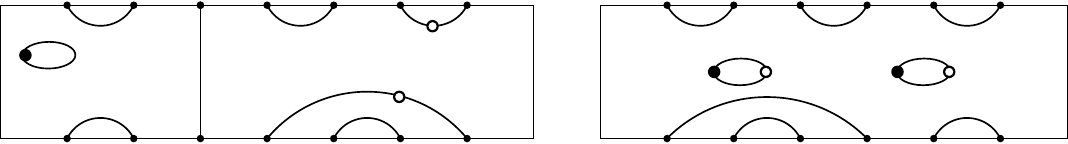}
\caption{Top basic diagrams. The non-propagating edges on the south face can be decorated.}
\label{basic}
\end{figure}

\begin{Proposition}\label{dimsimple}
If $D\in \dbD$ is an ALT-diagram different from the identity diagram $I$, then one of the following occurs:
\begin{enumerate}
\item $D$ is a top basic diagram;
\item there exists a simple edge in $D$ of type $U, N, P, S$ or $P^*$.
\end{enumerate}
\end{Proposition}

\begin{proof}
If $D\neq I$ is not top basic, then there exists at least one non-propagating edge on the north face joining two consecutive nodes. Necessarily one of these is a simple edge $e=\{i, i+1\}$, in fact, by (D1) we cannot have a $\circ$-decoration on a non-propagating edge to the left of an edge with a $\bullet$-decoration. If $e$ is of type $U$, $N$, $P$, $S$ or $P^*$ then (2) is satisfied. Now suppose that $i$ is the maximum index such that $e$ as well as all the simple edges on its left are not of any type. Consider the edge $e'$ leaving node $i+2$.
\begin{itemize}
    \item[(a)] Suppose $e'=\{i+2, i+k\}$. If $k>3$ then there exists a simple edge of type $N$ to the right of $e$, so assume $k=3$. The edge $e'$ cannot be decorated with a $\bullet$, otherwise $e$ would be of type $U_R$, and $i\ne n$, otherwise $D$ would be top basic. Then either $e'$ is simple, so by maximality it is of type $U$, $N$, $P$, $S$ or $P^{\circ}_L$, or it is decorated with a $\circ$. In the latter case, by (D1), on the right of $e'$ there can be only non-propagating edges. Hence, there exists another simple edge and it has to be of type $N$ or $U_L$ (see Figure \ref{dimtype} (C)). 

    \item[(b)]  Suppose $e'=\{i+2, (i+2)'\}$. If $e'$ is undecorated, since $D$ is not top basic, there must be another simple edge on its right which, by maximality, has to be of some type. The edge $e'$ cannot be decorated with a $\bullet$
   otherwise $e$ would be of type $S_R$. If $e'$ is decorated with a $\circ$, then $i\ne n$ otherwise $D$ would not be admissible, and by (D1) to the right of $e'$ there can only be non-propagating edges. Then there exists a simple edge of type $N$, $S_L$ or $U_L$, see Figure \ref{dimtype} (B).
   
   \item[(c)] Suppose $e'=\{i+2, m'\}$ not vertical. Then $m> i+2$ and, since the non-propagating edges right to $i+2$ can be decorated at most with a $\circ$, then there exists a simple edge of type $P_L$, $N$, or $U_L$ (see Figure \ref{dimtype} (A)).
\end{itemize}

\end{proof}

\begin{figure}[hbtp]
\centering
\includegraphics[scale=0.6]{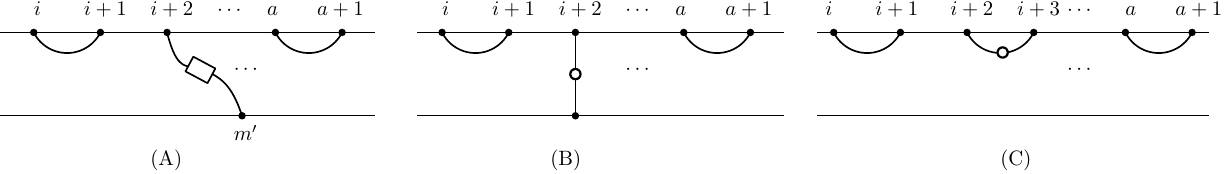}
\caption{Configurations in the proof of Proposition \ref{dimsimple}.}
\label{dimtype}
\end{figure}

\begin{Definition}[Suitable edges and neighbors]\label{def:suitable}
Let $D$ be an ALT-diagram.
\begin{itemize}
\item[(1)] Assume $D$ is top basic:
\begin{itemize}
    \item[(a)] if $D$ has at least one $ \locb$, then we say that any simple edge $e$ is \textit{suitable} and the leftmost $\locb$ is the \textit{neighbor of $e$};
    \item[(b)] if $D$ has a $\lob$, then $\{1,2\}$ is \textit{suitable} and the $\lob$ is its \textit{neighbor};
    \item[(c)] if $D$ has a $\loc$, then $\{n+1,n+2\}$ is \textit{suitable} and the $\loc$ is its \textit{neighbor};
    \item[(d)] if $D$ has no loops and no decorated edges on the south face, then any simple edge $e=\{i,i+1\}$ is \textit{suitable} and its \textit{neighbor} is the highest edge below $e$ intersected by the vertical line $i + 1/2$;
    \item[(e)] if $D$ has no loops and at least one decorated edge on the south face, then a simple edge $e=\{i, i+1\}$ is \textit{suitable} if the vertical line $i+1/2$ intersects in the south face either an edge decorated with both a $\bullet$ and a $\circ$, or the rightmost edge decorated with a $\bullet$, or the leftmost edge decorated with a $\circ$. The edge intersected by this vertical line is called \textit{neighbor of $e$}.
\end{itemize}
\item[(2)] If $D$ is not top basic, then any simple edge of type $U$, $N$, $P$, $S$ or $P^*$ is \textit{suitable}. Its \textit{neighbor} is the unique edge $f$ defined in Definition~\ref{deftypes}, except in type $N$, where the \textit{neighbor} is the highest non-propagating edge $f=\{j,k\}$ below $e$, and in type $P_R^\bullet$ (respectively $P_L^\circ$), where the \textit{neighbor} is the $\lob$ (respectively $\loc$).

\end{itemize}

\end{Definition}

Note that in an ALT-diagram $D$ a suitable edge always exists by definition and Proposition~\ref{dimsimple}. 

\begin{Example}
Both edges $\{3,4\}$ and $\{5,6\}$ in the top basic diagram of Figure \ref{suitable}(A) are suitable and their common neighbor is $\{3',6'\}$, while $\{1,2\}$ is not suitable. In the diagram in Figure \ref{suitable}(B), $\{1,2\}$ and $\{5,6\}$ are suitable since they are of type $S_R$ and $P_L$ respectively. In Figure \ref{suitable}(C) all the simple edges are suitable with common neighbor the leftmost $\locb$.
\begin{figure}[hbtp]
\centering
\includegraphics[height=20mm, keepaspectratio]{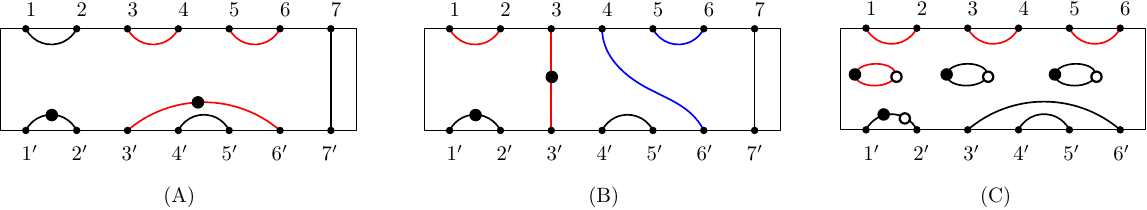}
\caption{Suitable edges and neighbors.}
\label{suitable}
\end{figure}
\end{Example}

 Now, we describe a procedure called \textit{cut and paste operation}, $\cp$-operation for short, that starting from any suitable edge $e=\{i,i+1\}$ produces an admissible diagram $D'$ such that $D=D_i D'$, or $D=D_0 D'$ if $e$ is of the form $(\stackrel{\bullet}{\smile})$  or $D=D_{n+2}D'$ if $e$ is of the form $(\stackrel{\circ}{\smile})$. Recall that $D_i$, with $i\in\{0, 1 \ldots, n+2\}$ are the simple diagrams described in Figure \ref{simple}. The idea is quite easy: we replace (except in one case) the suitable edge $e$ with two new non-intersecting edges, as sketched in Figure~\ref{cutpaste}. If $e$ is decorated, then we displace the decorations of the neighbor edge of $e$ in an ad hoc way obtaining a new admissible diagram. 

\begin{figure}[hbtp]
\centering
\includegraphics[scale=0.6]{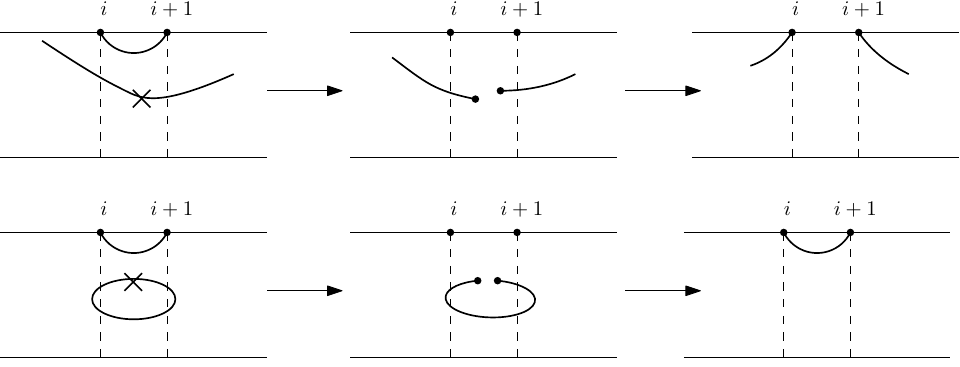}
\caption{Step (cp2) of the cut and paste operation.}
\label{cutpaste}
\end{figure}

\begin{Definition}[Cut and paste operation]\label{definition:cutandpaste}
Let $D$ be an ALT-diagram with a suitable edge $e=\{i, i+1\}$.
\begin{enumerate}
\item[(cp1)] Delete the simple edge $e$.
\item[(cp2)] Cut the neighbor of $e$ and join the two free endpoints of the cut edge to the nodes $i$ and $i+1$ as to obtain two new non-intersecting edges or the edge $\{i,i+1\}$, see Figure~\ref{cutpaste}. 
\item[(cp3)] If the neighbor edge of $e$ is decorated, then: 
\begin{itemize}
    \item[(1)] distribute its decorations on the new edges as in Figure~\ref{cp-general}, if $e$ is one of the depicted cases and denote by $\cp_e(D)$ the diagram obtained from $D$ applying the previous steps;
    \item[(2)] $\cp_e(D)=F(\cp_{F(e)}(F(D)))$  otherwise.
\end{itemize}
\end{enumerate}
\end{Definition} 

\begin{figure}[hbtp]
\centering
\includegraphics[scale=0.7]{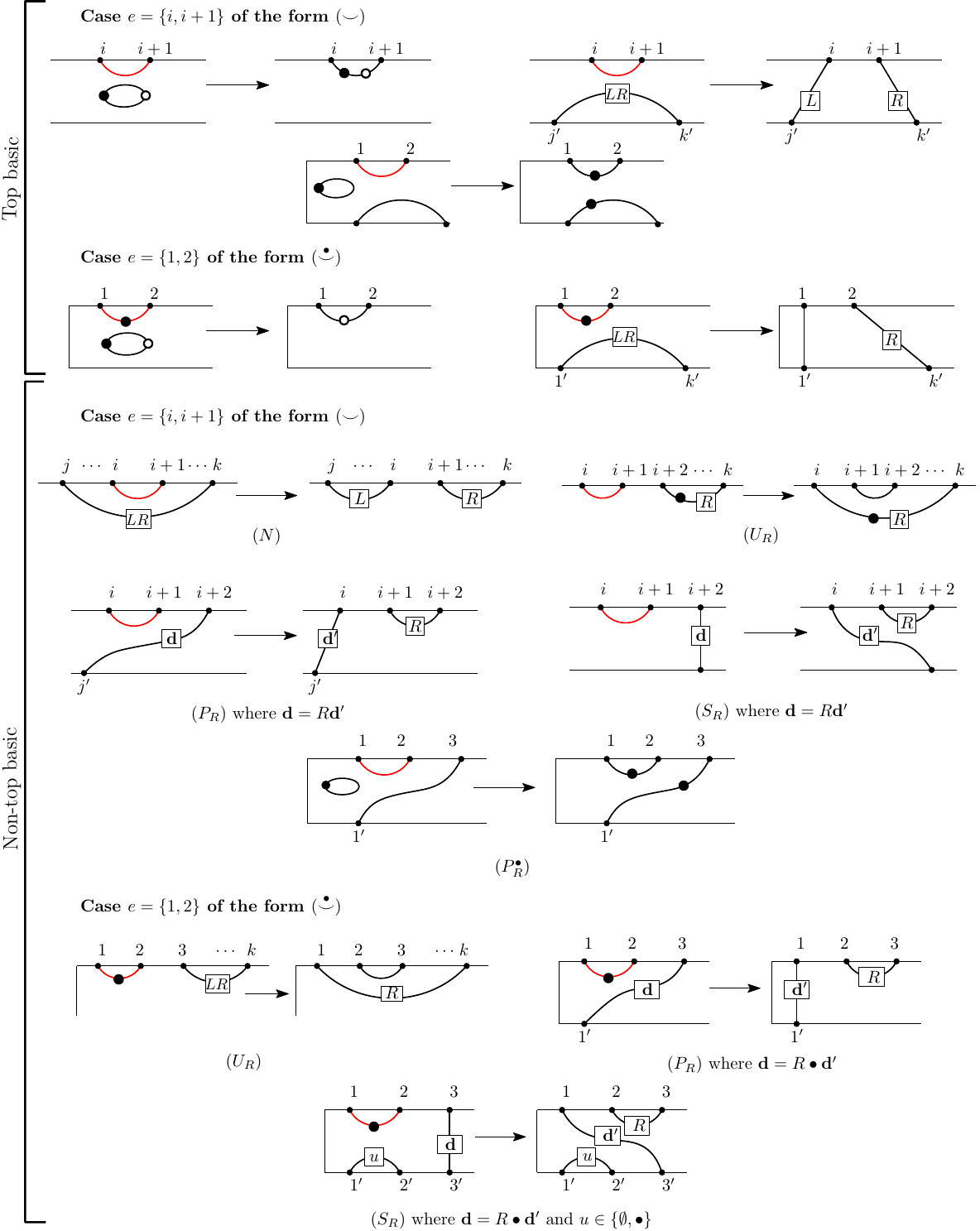}
\caption{Local cp-operations, where $L\in\{\emptyset, \bullet\}$ and $R\in \{\emptyset, \circ\}$.}
\label{cp-general}
\end{figure}

It is not hard to see that given an ALT-diagram $D$ with a suitable edge $e$, either $e$ or $F(e)$ appears  in Figure~\ref{cp-general}. In Appendix~\ref{A} some detailed examples of applications of the cp-operations are given. In Figure~\ref{cpF}, we develop an application of the cp-operation for a suitable edge not appearing in Figure~\ref{cp-general}.

\begin{figure}[hbtp]
\centering
\includegraphics[scale=0.7]{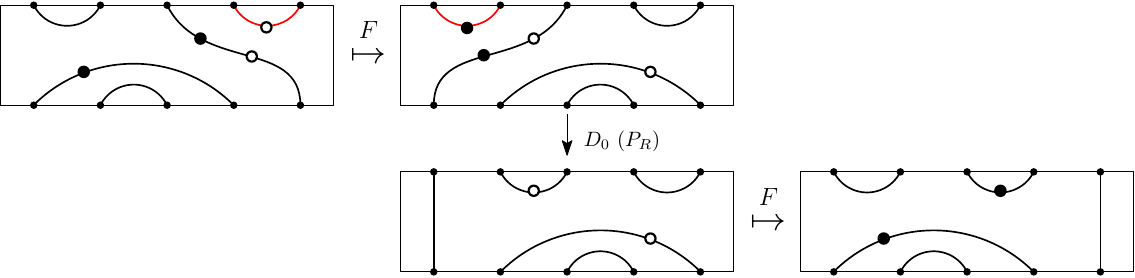}
\caption{Example of cp-operation for a simple edge $e$ of type $P_L$. First conjugate the diagram, then apply the cp-operation to $F(e)$ (of type $P_R$) and finally apply again the conjugate. }
\label{cpF}
\end{figure}

\smallskip

Now we show that the newly obtained diagram $\cp_e(D)$ is admissible and enjoys a uniqueness property.

\begin{Lemma}\label{lemma:cplength}
Let $e=\{i,i+1\}$ be a suitable edge of an ALT-diagram $D$ in $\dbD$. Then 
\begin{enumerate}
    \item $\cp_e(D)$ is admissible;
    \item $D=D_0 \,\cp_e(D)$ if $e$ is of the form $(\stackrel{\bullet}{\smile})$, $D=D_{n+2} \,\cp_e(D)$ if $e$ is of the form $(\stackrel{\circ}{\smile})$, $D=D_i \, \cp_e(D)$ otherwise.
    \item $\ell(\cp_e(D))=\ell(D)-1$.
\end{enumerate}
\end{Lemma}

\begin{proof}
By (cp2), the edges of $\cp_e(D)$ do not intersect. It is not difficult to verify that all the conditions in Definition \ref{def:admissible} hold. In fact, in $\cp_e(D)$ no new loops occur, so (A1) is satisfied. In Figure \ref{cp-general}, we describe case by case how the decorations distribute in the diagram after the cp-operation. In these procedures, either the number of the decorations is the same, or it varies of 2, so (A2) holds. Finally, in the situations in Figure \ref{cp-general} where $\{1,1'\}$ occurs in $\cp_e(D)$ or $\ab(\cp_e(D))=1$, (A3) and (A4) hold. By diagram concatenation, (2) easily follows.

By Lemma \ref{rem:length-conj}, $\ell(F(D))=\ell(D)$. Moreover, by (cp3) in Definition \ref{definition:cutandpaste}, it suffices to consider the suitable edges in Figure~\ref{cp-general} since each remaining case is obtained as a conjugate of one of these.

 Statement (3) follows from a case-by-case verification based on the computation of the difference between the lengths of $D$ and $\cp_e(D)$. We only need to consider the edges involving the nodes:
 \begin{itemize}
     \item[($e_1$)] $i$ and $i+1$ when $i\ge 1$ and $e$ is of the form $(\smile)$;
     \item[($e_2$)]  1 and 2 when $i=1$ and $e$ is of the form $(\stackrel{\bullet}{\smile})$;
 \end{itemize}
   since all the other edges in $D$ and $\cp_e(D)$ coincide.
More precisely, one of the following occurs:
\begin{itemize}
    \item[(a)] $\weight(\cp_e(D))=\weight(D)$ and $\nu(\cp_e(D))=\nu(D)-2$;
    \item[(b)] $\weight(\cp_e(D))=\weight(D)-2$ and $\nu(\cp_e(D))=\nu(D)+2$;
    \item[(c)] $\weight(\cp_e(D))=\weight(D)-1$ and $\nu(\cp_e(D))=\nu(D)$,
\end{itemize}
from which $\ell(\cp_e(D))=\ell(D)-1$, proving the result. For instance, case (c) occurs only when $e$ is of type $P^*$ or when in a top basic diagram there is a loop edge. 

The computations for the cases $N, U_R, P_R, P_R^\bullet,  S_R$ are detailed in Table~\ref{tabular} in the Appendix~\ref{A}. Now we analyse explicitly the top basic cases referring to Figure \ref{cp-general}. If $D$ has a $\locb$, then $\weight(\locb)_D=n+1$ while $\weight(i,i+1)_{\cp_e(D)}=n$ and $\nu(\cp_e(D))=\nu(D)$. 
If $D$ has a $\lob$, then $\nu(D)=\nu(\cp_e(D))$ and $\weight(D)=\weight(\cp_e(D))+1$ since $\weight(\lob)=1$ and the other edges have weight equal to zero.
In the second case of the first row, if $\bf d=\bullet^\alpha \circ^\beta$, where $\alpha, \beta \in \{0,1\}$, then $\weight(j',k')_D=\alpha(j-1)+\beta(n+2-k)$, while $\weight(i,j')_{\cp_e(D)}=\alpha(j-1)$ and $\weight(i+1, k')_{\cp_e(D)}=\beta(n+2-k)$. Hence $\weight(D)=\weight(\cp_e(D))$ and $\nu(D)=\nu(\cp_e(D))+2$ since the vertical line $i+1/2$ in $\cp_e(D)$ does not intersect the edges $\{i, i+1\}$ and $\{j', k'\}$ as in $D$.  
\end{proof}

 We are able to prove the following.

\begin{Proposition}\label{unicity}
Let $e$ be a suitable edge of an ALT-diagram $D$. Then $\cp_e(D)$ is the unique admissible diagram satisfying the conditions (1), (2) and (3) of Lemma~\ref{lemma:cplength}.
\end{Proposition}

\begin{proof}
As in the proof of Lemma~\ref{lemma:cplength}, by Definitions~\ref{definition:cutandpaste} and \ref{def:flipped}, it is enough to consider the situations described in Figure~\ref{cp-general}.
We first assume that $e=\{i, i+1\}$ is undecorated and let $D'$ be an admissible diagram such that $D=D_iD'$ and $\ell(D')=\ell(D)-1$. Now we show that $sh(D')$ must be equal to $sh(\cp_e(D))$.
\smallskip

If $D$ is a top basic diagram, then we have the following three situations. 
 \begin{itemize}
\item[(B1)]  Suppose a $\locb$ occurs in $D$. If $sh(D')\neq sh(\cp_e(D))$, then the edge $\{i, i+1\}$ is not in $D'$, and since $D$ is top basic with a $\locb$, the two nodes $a$ and $b$ connected respectively to $i$ and $i+1$ in $D'$ are on the north face, and
$b<a<i$ or $i+1<b<a$. In one case, $\{i, a\}$ and $\{b, i+1\}$ intersect the vertical line $(i-1)+1/2$, while in the other, $\{i, a\}$ and $\{i+1, b\}$ intersect the vertical line $(i+1)+1/2$, so $\nu(D')> \nu(D)$. These two edges have weight 0. In fact they are all undecorated, except  $\{1, i+1\}$ which is decorated with a $\bullet$ when $\{b,a\}=\{1,2\}$ in $D$ is decorated with a $\bullet$, or $\{i, n+2\}$ which is decorated with a $\circ$ when $\{b,a\}=\{n+1, n+2\}$ in $D$ is decorated with a $\circ$. Note also that the number of $\locb$ is the same in $D$ and $D'$. It follows that $\weight(D')=\weight(D)$, so $\ell(D')\geq \ell(D)+1$, against the hypothesis. Hence $sh(D')=sh(\cp_e(D))$. 

\item[(B2)] Suppose that a $\lob$ appears in $D$. If $sh(D')\neq sh(\cp_e(D))$ then the edge $\{1,2\}$ does not occur in $D'$. Let $a^*$ and $b^*$ be the nodes in $D'$ connected to $1$ and $2$ respectively. If $a,b>2$, both the edges $\{1, a^*\}$ and $\{2, b^*\}$ intersect the vertical line $2+1/2$. Note that $D'$ necessarily has a $\lob$ and so it cannot contain any other $\bullet$ decoration. Moreover note that $\{1, a^*\}$ is decorated with a $\circ$ if and only if $\{b^*, a^*\}=\{n+1, n+2\}$ is decorated with $\circ$ in $D$. The edge $\{2, b^*\}$ is decorated with a $\circ$ if and only if $\{a^*,b^*\}=\{a', b'\}$ in $D$ is decorated with a $\circ$. Therefore, by Remark \ref{rem:endpoints}, $\weight(2, b')_{D'}=\weight(a',b')_D$. Hence $\ell(D')>\ell(D)$ against the hypothesis and  $sh(D')=sh(\cp_e(D))$. If $a^*=1'$, then $D'$ is a LP-diagram with $j_\ell \ge 2$. It is easy to see that $\ell(\widehat{D'})=\ell(D)-2$, hence $\ell(D')=2j_\ell -1 + \ell(\widehat{D'})=2 j_\ell -1 + \ell(D) -2\geq \ell(D)+1$, against the hypothesis, so also in this case $sh(D')=sh(\cp_e(D))$.

\item[(B3)] Now, suppose that there are no loop edges in $D$; see the first row of Figure~\ref{cp-general}, right. If $sh(D')\neq sh(\cp_e(D))$, then the edge $\{j',k'\}$ occurs in $D'$ and it is decorated as in $D$. Let $a^*$ and $b^*$ be the nodes in $D'$ connected to $i$ and $i+1$ respectively. As in (B1), $\weight(D')=\weight(D)$ and $a,b<i$ (respectively $a,b>i+1$), so the edges $\{a^*,i\}$, $\{b^*, i+1\}$ (respectively $\{i, a^*\}$, $\{i+1, b^*\}$) intersect the vertical line $(i-1)+1/2$ (respectively $(i+1)+1/2$), then $\ell(D')\geq \ell(D)+1$, against the hypothesis. Hence again $sh(D')=sh(\cp_e(D))$. 
\end{itemize}

\smallskip

Now assume $D$ is not top basic. There are several cases to consider.
\begin{itemize}
\item[($N$)] Let $e$ be of type  $N$. Suppose by contradiction that $sh(\cp_e(D))\neq sh(D')$. Then at least one of the two edges
 $\{j,i\}$ and $\{i+1,k\}$ is not in $D'$, and  suppose $\{j,i\}$ is not in $D'$ (the other case being analogous). Set $t^*\neq i$ the node connected to $j$ in $D'$. Since $j$ and $t^*$ are not involved in the product $D_iD'=D$, the edge $\{j,t^*\}$ must occur and be decorated as in $D$ and so $t^*=k$. If in $D'$, $i$ is connected to $a\neq i+1$ and $i+1$ is connected to $b\neq i$, then the edge $\{a,b\}$ is in $D=D_iD'$, therefore either $a,b<i$ or $a,b>i+1$ since $\{j,k\}$ is the neighbor of $e$. In the two situations, either $\{a,i\}$ and $\{b, i+1\}$ intersect the vertical line $(i-1)+1/2$ or $\{i,a\}$ and $\{i+1, b\}$ intersect the vertical line $(i+1)+1/2$. Since such edges are undecorated, in both cases, $\ell(D')\geq \ell(D)+1$ against the hypothesis. Hence $sh(D')=sh(\cp_e(D))$.

\smallskip

\item[($P$)] Let $e$ be of type  $P_R$. In $D'$, for parity reasons, $j'$ is not connected to $i+1$. As before, suppose by contradiction that $sh(D')\neq sh(\cp_e(D))$, then at least one edge between $\{i+1,i+2\}$ and $\{i, j'\}$ is not in $D'$. Suppose $\{i+1, i+2\}$ is not in $D'$ (the other case being analogous), then the edge $\{i+2, j'\}$ occurs in $D'$. In $D'$, $i$ is connected to a node $a^*\neq i+1$ and $i+1$ is connected to a node $b^*\neq i$, with $a^*, b^*<i$. Then the vertical line $(i-1)+1/2$ does not intersect the edge $\{a^*, b^*\}$ in $D$ while it intersects both edges $\{a^*, i\}$ and $\{b^*, i+1\}$  in $D'$. Furthermore, the edge $\{a^*, b^*\}$ in $D$ could be decorated with at most a $\bullet$. In this case, the $\bullet$ should appear in $\{a^*, i\}$ or $\{b^*, i+1\}$ in $D'$, and by the definition of weight, one can check that $\mathsf{w}(a^*, b^*)_{D}\leq \mathsf{w}(a^*, i)_{D'}+\mathsf{w}(b^*, i+1)_{D'}$ leading to $\ell(D')\geq \ell(D)+1$ against the hypothesis. Therefore $sh(\cp_e(D))=sh(D')$.

\smallskip

\item[($U$)] Let $e$ be of type  $U_R$. As before, suppose by contradiction that $sh(D')\neq sh(\cp_e(D))$, then necessarily one between $\{i,k\}$ and $\{i+1,i+2\}$ does not occur in $D'$. Suppose that $\{i,k\}$ is not in $D'$ (the other case being analogous). In $D'$, $k$ must be connected to $i+2$ and decorated with $\bullet R$ as in $D$. Now assume that $i$ is connected to $a\neq i+1$ and $i+1$ to $b\neq i$. Since $\{i+2, k\}$ contains a $\bullet$ this forces $b<a<i$. Now the edges $\{b, i+1\}$ and $\{a, i\}$ intersect the vertical line $(i-1)+1/2$ in $D'$ while no edges in the north face of $D$ intersect this line. Since $\{b,a\}$ can be decorated  with at most a $\bullet$, by Remark \ref{rem:endpoints}, $\weight(b,a)_D=\weight(b, i+1)_{D'}$. Hence $\ell(D')\geq \ell(D)+1$, against the hypothesis, then $sh(D')=sh(\cp_e(D))$. 

\smallskip

\item[($S$)] Let $e$ be of type  $S_R$. As before, suppose by contradiction that $sh(D') \neq sh(\cp_e(D))$, then at least one between $\{i+1,i+2\}$ and $\{i,(i+2)'\}$ does not occur in $D'$. Suppose that $\{i+1, i+2\}$ is not in $D'$ (the other case being analogous), then the edge $\{i+2, (i+2)'\}$ must occur and be decorated with \textbf{d} as in $D'$. The same argument used in $U$ shows that also in this case $\ell(D')\geq \ell(D)+1$, against the hypothesis. Then $sh(D')=sh(\cp_e(D))$.
\smallskip

\item[($P_R^\bullet$)] Let $e=\{1,2\}$ be of type $P_R^\bullet$. Again, suppose by contradiction that $sh(D') \neq sh(\cp_e(D))$. Then $\{1,2\}$ is not contained in $D'$, thus necessarily the edges $\{1,1'\}$, $\{2,3\}$ 
 and a $\lob$ occur in $D'$. Now, $D'$ is a LP-diagram with $j_\ell =2$ hence $\ell(D')= 3 + \ell(\widehat{D'})$. Moreover as in (B2) we have $\ell(\widehat{D'})=\ell(D)-2$, so $\ell(D')=\ell(D)+1$, against the hypothesis. Therefore  $sh(D') = sh(\cp_e(D))$.

\end{itemize}
 In Lemma \ref{lemma:cplength}, we showed that $\ell(\cp_e(D))=\ell(D)-1$ and we just proved that if $D=D_iD'$ and $\ell(D')< \ell(D)$, then $sh(D') = sh(\cp_e(D))$. Now we have to show that among the ways of splitting the decorations on the edges of $D'$ in an admissible way, the unique for which $\ell(D')=\ell(D)-1$ yields is the one provided by the cp-operation. In all the top basic cases, $U_R$ and $P_R^\bullet$ there is only one admissible way. In cases $N$, $P_R$ and $S_R$ there are at most three possibilities and in those that differ from the cp-operation way, only one decoration can be displaced. In these cases, the total weight increases and the length of $D'$ is not the right one. Hence the only possibility is $D'=\cp_e(D)$. 

To complete the proof one should analyze the suitable edge $(\stackrel{\bullet}{\smile})$. However,
the only differences with respect to the treated cases are extra occurrences of $\bullet$, and the proof is analogous.
\end{proof}

\begin{Proposition}\label{unicityP}
Let $D$ be a P-diagram in $\dbD$. Then there exist a simple edge $e=\{i,i+1\}$ and a unique admissible diagram $D'$ such that $\ell(D')=\ell(D)-1$ and $D=D_0 D'$ if $e$ is of the form $(\stackrel{\bullet}{\smile})$, $D=D_{n+2}D'$ if $e$ is of the form $(\stackrel{\circ}{\smile})$, $D=D_i D'$ otherwise.
\end{Proposition}

\begin{proof}
Assume that $D$ is of type LP. Let us suppose that 
\begin{itemize}
    \item[($\star$)] $D$ has a simple edge $\{j_\ell, j_{\ell}+1\}$ of type $P_R$ or $\widehat{D}$ is top basic.
\end{itemize}
 where $\widehat{D}$ is the inner alternating diagram of Definition~\ref{def:altportion} and $j_\ell>1$.
Consider the diagram $D'$ obtained from $D$ by deleting $\{j_\ell, j_\ell+1\}$, cutting the vertical edge $\{j_\ell-1, (j_\ell-1)'\}$ and joining $j_\ell-1$ to $j_\ell$ and $(j_\ell-1)'$ to $j_\ell+1$ and denote this procedure by \textit{$K_L$-operation} (see Appendix~\ref{A} for an example of application of the $K_L$-operation). The diagram $D'$ starts from the left with one $\lob$, $j_\ell-2\geq 0$ undecorated vertical edges and satisfies $D=D_{j_\ell}D'$. 
Note that $\nu_{j_\ell-1}(D')=\nu_{j_\ell-1}(D)+2$ and $\nu_k(D')=\nu_k(D)$ for all $k\neq j_\ell-1$. It follows that $\ell(\widehat{D})=\ell(\widehat{D'})-1$. 
When $j_\ell-2= 0$, we set $\widehat{D'}$ as $D'$ without the unique $\lob$.
Hence, $\ell(D')=2(j_\ell-1)-1 + \ell(\widehat{D'})=2j_\ell - 3 + \ell(\widehat{D})+1 = \ell(D)-1$. It remains to show that $D'$ is unique. 
\smallskip

Let $D''$ be an admissible diagram such that $D=D_{j_\ell} D''$, $\ell(D'')=\ell(D)-1$ and suppose by contradiction that $sh(D'')\neq sh(D')$. Since the shape are different, at least one of the two the edges $\{j_\ell-1, j_\ell\}$ and $\{j_\ell+1, (j_\ell-1)'\}$ does not occur in $D''$.  Since $D$ contains $\{j_\ell-1, (j_\ell-1)'\}$, and $D=D_{j_\ell} D''$ then such edge has to be also in $D''$, that turns out to be a LP-diagram having at least $j_\ell-1$ consecutive vertical edges from the left. If $j_\ell=n+1$, then either $D''=D$, or all the edges of  $D''$ from 1 to $n$ are vertical. Both cases are not possible: the first is clear, while in the second we get a non admissible diagram since condition (D0) of Section \ref{sec:undecorated} is not satisfied. Hence, we have $j_\ell<n+1$. 
\smallskip

By Definition \ref{deflength}, $\ell(D'')= 2k -1 + \ell(\widehat{D''})$, where $k-1\geq j_\ell -1$ is the number of consecutive undecorated vertical edges from the left.

Let $a^*$ be the node connected to $j_\ell$ and $b^*$ the one connected to $j_\ell +1$. It is clear that $D$ contains the edge joining $a^*$ and $b^*$. One of the following three cases can occur in $D''$, as depicted in Figure \ref{fig:cp_P}:
\begin{itemize}
    \item[(1)] $a^*=j_\ell'$, $b^*=(j_\ell +1)'$ and the edges are undecorated;
    \item[(2)] $a^*=j_\ell'$ and $b>j_\ell +1$ or
    $b^*=(j_\ell +1)'$ and $\{j_\ell +1, (j_\ell +1)'\}$ is decorated with a $\circ$;
    \item[(3)] $a,b>j_\ell +1$.
\end{itemize}

\begin{figure}
    \centering
    \includegraphics[scale = 0.6]{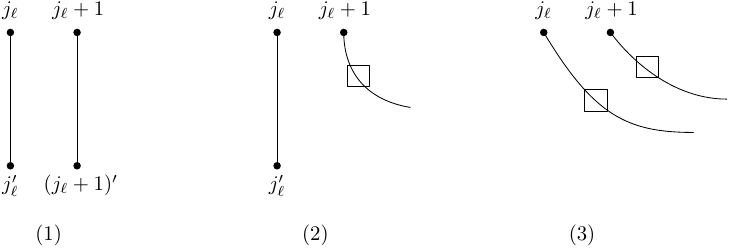}
    \caption{Cases in the proof of Proposition \ref{unicityP}. The boxes can be either empty or filled with a $\circ$. In (3), at least one of the two boxes is empty.}
    \label{fig:cp_P}
\end{figure}

 First note that if (1) or (3) occurs, then by ($\star$), $\widehat{D}$ is top basic, since $j_\ell'$ is not joined to $j_\ell +2$. In case (1), $D''$ has at least $j_\ell +1$ consecutive undecorated vertical edges from the left. It implies that in $D$, if $k-1>j_\ell +1$, the edges leaving the nodes from $j_\ell +2$ to $k-1$ are all vertical. Moreover, the portion of $D$ starting from the node $k$ is equal to $\widehat{D''}$, so $\ell(\widehat{D})=\ell(\widehat{D''})+1$. Therefore, $\ell(D'')\geq 2(j_\ell +2) -1 + \ell(\widehat{D''})= 2j_\ell +3 + \ell(\widehat{D})-1 = \ell(D)+3$, which contradicts the hypothesis. In case (2), comparing the lengths of $\widehat{D''}$ and $\widehat{D}$, it is not difficult to see that $\ell(\widehat{D''})=\ell(\widehat{D})-1$. In fact, note that the size of the box of $\widehat{D''}$ is one less the size of the box of $\widehat{D}$, and if the edge $\{j_\ell +1, b^*\}$ is decorated with a $\circ$, then $\weight_{D''}(\{j_\ell +1, b^*\})=\weight_D(\{j_\ell', b^*\})$. Hence, $\ell(D'')=2j_\ell +1 + \ell(\widehat{D})-1=\ell(D)+1$, which again contradicts the hypothesis. In case (3), the standard boxes of $\widehat{D}$ and $\widehat{D''}$ have the same size, so similarly to the arguments in the proof of Proposition \ref{unicity} for (B3), we have $\nu(\widehat{D''})\geq \nu(\widehat{D})$, so $\ell(D'')>\ell(D)$.

Hence $sh(D'')=sh(D')$. This shows that $D''=D'$ since the edges involving $j_\ell$ and $j_\ell+1$ in $D$ are undecorated, being $1< j_\ell < n+1$.
\smallskip 

If $(\star)$ is not satisfied, then by Proposition \ref{dimsimple}, in $\widehat{D}$ we can find a suitable edge $\widehat{e}$ of any type except $U_R$, $S_R$ and $P_R^\bullet$. It implies that $\widehat{e}$ does not involve the node $\widehat{1}$. Now we can apply the cp-operation to $\widehat{D}$ on $\widehat{e}$. Now, the diagram $D'$ obtained by putting side by side from the left one $\lob$, $j_\ell-1$ undecorated vertical edges and $\cp_{\widehat{e}}(\widehat{D})$, satisfies the thesis and it is unique by Proposition \ref{unicity}.
\smallskip 

If $D$ is of type RP, $F(D)$ is a P-diagram of type LP and so we just proved that there exist $i$ and $D'$ with $\ell(D')=\ell(D)-1$, such that $F(D)=D_i D'$, from which $D=F(F(D))=F(D_i)F(D')$, where $F(D_i)=D_{n+3-i}$.
\smallskip 

Finally, if $D$ is of type LRP, the thesis is obtained combining the two procedures above.
\end{proof}

%%%%%%%%%%%%%%%%%%%%%%%%%%%%%%%%%%%%%%%%%%%%%%%%%%%%
\section{The algebra isomorphism}
\label{sec:isom}
%%%%%%%%%%%%%%%%%%%%%%%%%%%%%%%%%%%%%%%%%%%%%%%%%%%%%

In this section we relate the admissible diagrams to the basis elements of the Temperley--Lieb algebra $\TL(\Dn)$.
\smallskip

Let ${\bf w}=s_{i_1}s_{i_2}\cdots s_{i_k}$ be a reduced expression of $w\in W$, and set
$$D_{\bf w}:=D_{i_1}D_{i_2}\cdots D_{i_k}.$$

If $w\in \FC(\Dn)$, we can define $D_w:=D_{\bf w}$, where $\bf w$ is any reduced expression of $w$, since $D_{w}$ does not depend on the chosen reduced expression of $w$.
We define the $\Zd$-algebra homomorphism 
\begin{eqnarray}
    \tilde{\theta}_D: \TLD & \longrightarrow & \DD \\
    b_i &\mapsto & D_i \nonumber
\end{eqnarray} 
for all $i=0,\ldots,n+2$. 
Clearly, $\tilde{\theta}_D$ is surjective and maps the monomial basis element $b_w$ into the diagram $D_w$. Our goal is to show that this map is bijective. The proof will be based on a factorization of any admissible diagram into simple diagrams that we now present.

\subsection{Factorizations of admissible diagrams}

\begin{Lemma}\label{ad1}
Let $D$ be a PZZ-diagram or a P-diagram with $\ab(D)=1$. Then $D$ can be written as a finite product of simple diagrams.
\end{Lemma}

\begin{proof}
To prove this result, we provide an explicit factorization for any PZZ-diagram and any P-diagram with $\ab(D)=1$. Let $D$ be a PZZ-diagram, $\{i, i+1\}$ the non-propagating edge on the north face and $\{j', (j+1)'\}$ the non-propagating edge on the south face. 
Define $z_1:=s_0s_1s_2\cdots s_n$ and $z_2:=s_{n+2}s_{n+1}s_n\cdots s_2$. By direct computations, $D$ can be factorized as one of the following products depending on the number $\mathtt{r}$ of $\loc$ and on its type, as defined in Definition~\ref{defPZZ}:
\begin{align*}
    \langle\circ \, \circ\rangle \ &D_0^{\alpha}(D_i^{1-\alpha}D_{i+1}\cdots D_n)^{1-\delta_{i, n+1}}(D_{z_2}D_{z_1})^{\mathtt{r}-1}D_{n+2}D_{n+1}(D_n \cdots D_j^{1-\beta}D_0^{\beta})^{1-\delta_{j,n+1}}\\
    \langle\bullet \, \circ\rangle \ &D_{n+2}^{\eta}(D_i^{1-\eta}D_{i-1}\cdots D_2)^{1-\delta_{i,1}}D_{z_1}(D_{z_2}D_{z_1})^{\mathtt{r}-1}D_{n+2}D_{n+1}(D_n \cdots D_j^{1-\beta}D_0^{\beta})^{1-\delta_{j,n+1}}\\
    \langle\circ \, \bullet\rangle \ &D_0^{\alpha}(D_i^{1-\alpha}D_{i+1}\cdots D_n)^{1-\delta_{i, n+1}}(D_{z_2}D_{z_1})^{\mathtt{r}-1}D_{z_2}D_{0}D_{1}(D_2 \cdots D_j^{1-\epsilon}D_{n+2}^{\epsilon})^{1-\delta_{j,1}}\\
    \langle\bullet \, \bullet\rangle \ &D_{n+2}^{\eta}(D_i^{1-\eta}D_{i-1}\cdots D_2)^{1-\delta_{i,1}}(D_{z_1}D_{z_2})^{\mathtt{r}}D_{0}D_{1}(D_2 \cdots D_j^{1-\epsilon}D_{n+2}^{\epsilon})^{1-\delta_{j,1}}
\end{align*}
where $\delta_{x,y}$ is the Kronecker delta and the parameters $\alpha, \beta, \eta, \epsilon$ are defined as follows:
\begin{itemize}
    \item[] $\alpha =1$ if $i=1$ and $\{1,2\}$ is decorated with a $\bullet$, $\alpha=0$ otherwise;
    \item[] $\beta=1$ if $j=1$ and $\{1',2'\}$ is decorated with a $\bullet$, $\beta=0$ otherwise;
    \item[] $\eta =1$ if $i=n+1$ and $\{n+1,n+2\}$ is decorated with a $\circ$, $\eta=0$ otherwise;
    \item[] $\epsilon = 1$ if $j=n+1$ and $\{(n+1)', (n+2)'\}$ is decorated with a $\circ$, $\epsilon=0$ otherwise.
\end{itemize}

The case where $D$ is a P-diagram with $\ab(D)=1$ is easier. Using the same notation as above, we have that if $D$ is a LP-diagram, then a factorization is of the form
$$D_{n+2}^{\eta}D_i^{1-\eta}D_{i-1}\cdots D_2D_0D_1D_2 \cdots D_j^{1-\epsilon}D_{n+2}^{\epsilon}.$$

If $D$ is a RP-diagram, it has a factorization of the form
$$D_0^{\alpha}D_i^{1-\alpha}D_{i+1}\cdots D_nD_{n+1}D_{n+2}D_n \cdots D_j^{1-\beta}D_0^{\beta}.$$

\end{proof}

\begin{Example}
 Let $D$ be the PZZ-diagram of type $\langle \bullet \, \bullet\rangle$ on the left of Figure~\ref{PZZexample1}, then $D = D_2(D_{z_1}D_{z_2})^2D_0D_1D_2D_3D_4D_6$, where $D_{z_1}=D_0D_1D_2D_3D_4$ and $D_{z_2}=D_6D_5D_4D_3D_2$. 
\end{Example}

\begin{Lemma}\label{PZZheaps}
Let $w\in \FC(\Dn)$. If $w\in(\PZZ)$ then $D_w$ is a PZZ-diagram.
\end{Lemma}

\begin{proof}
The proof follows from a direct computation. We detail only one specific case being the others very similar. By Definition \ref{family}, we can assume that $w$ has a reduced expression of the form
\begin{equation*}
    \mathbf{w}=s_is_{i+1}\cdots s_n (s_{n+1}s_{n+2}s_n\cdots s_2s_0s_1s_2\cdots s_n)^h s_{n+1}s_{n+2}s_n \cdots s_j,
\end{equation*}
with $1\le i,j < n+1$. Let $z_1$ and $z_2$ as defined in Lemma \ref{ad1}, then $$D_w=D_iD_{i+1}\cdots D_n (D_{z_2} D_{z_1})^h D_{n+1}D_{n+2}D_n \cdots D_j$$ is a PZZ-diagram of type $\langle \circ \, \circ \rangle$, with $\{i,i+1\}$ as unique non-propagating edge in the north face, $\{j',(j+1)'\}$ as unique non-propagating edge in the south face, with $h+1$ occurrences of $\loc$ and $h$ of $\lob$. 
\end{proof}

Before proving the following results, we need to introduce a particular representation of a diagram $D_w$ for some $w\in \FC$, by explicitly describing its factorization into simple diagrams geometrically. 

\begin{Definition}\label{rem:simplerepr}
Let $\mathbf{w}=s_{i_1}s_{i_2}\cdots s_{i_m}$ be a reduced expression of $w\in W$. We call \textit{concrete simple representation} of $D_{\bf w}$, the concrete diagram $\mathcal{C}(D_{\bf w})$ that results from stacking standard representations of the simple diagrams $D_{i_1}, \ldots, D_{i_m}$ from top to bottom without deforming any of their edges or applying any relation among the decorations (see Figure \ref{simplerepres}). Recall the description of a standard representation of a diagram given in Section \ref{sec:undecorated}, after Figure \ref{7box}.
\end{Definition}

\begin{figure}[hbtp]
\centering
\includegraphics[scale=0.5]{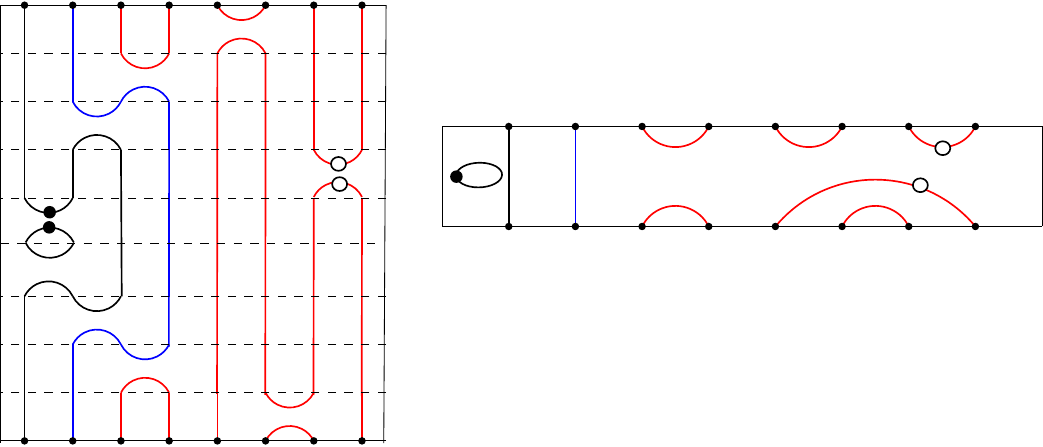}
\caption{The concrete simple representation $\mathcal{C}(D_{\bf w})$ with $\mathbf{w}=s_5s_3s_2s_8s_0s_1s_2s_3s_6$ in $\FC(\widetilde{D}_8)$ and its associated pseudo diagram $D_w$.}
\label{simplerepres}
\end{figure}

\begin{Lemma}\label{lemma:pattern}
    The element $w\in \FC(\Dn)$ if and only if, for any ${\bf w}\in \mathcal{R}(w)$, $\mathcal{C}(D_{\bf w})$ does not contain the patterns in Figure \ref{fig:patterns}.
\end{Lemma}

\begin{figure}[h]
    \centering
    \includegraphics[scale = 0.5]{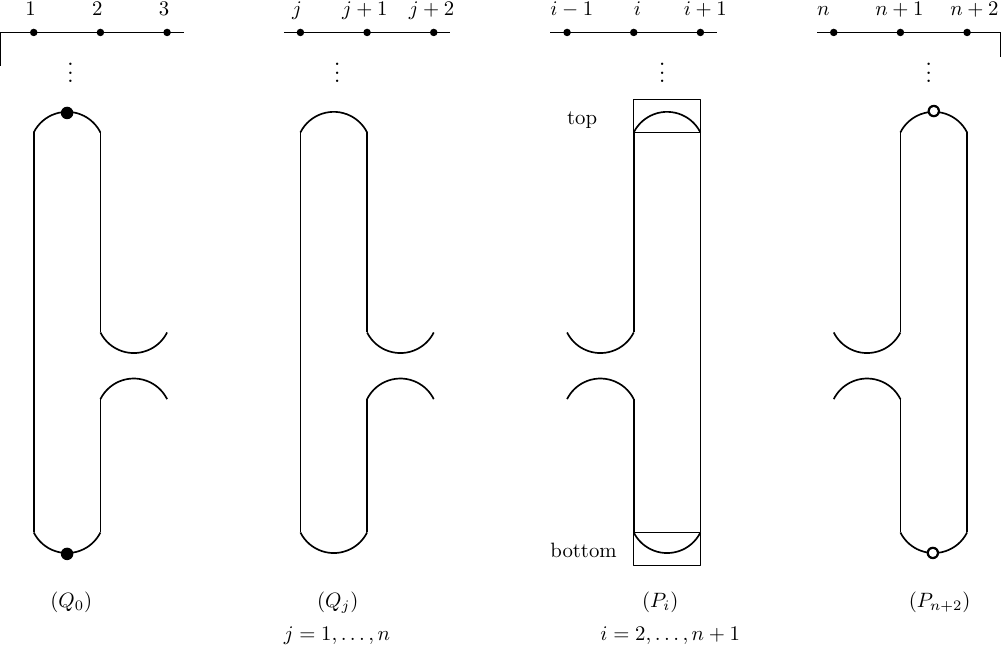}
    \caption{Forbidden patterns in a concrete simple representation of a FC element. }
    \label{fig:patterns}
\end{figure}

\begin{proof}
    Let $w\notin \FC(\Dn)$, then there exists a reduced expression $\bf w$ of $w$ containing the factor $s_i s_j s_i$ for some $s_i,s_j$ adjacent nodes in the Coxeter graph (Figure \ref{BDtilde}, left). Hence, $D_i D_j D_i$ in $D_{\bf w}$ gives rise to one of the patterns in Figure \ref{fig:patterns} in $\mathcal{C}(D_{\bf w})$.
    Vice versa, let $w\in \FC(\Dn)$ and let ${\bf w}\in \mathcal{R}(w)$ such that $\mathcal{C}(D_{\bf w})$ contains the pattern $P_i$, with $i=2,\ldots, n+2$. By construction, the diagrams contributing to the top and bottom part of the pattern $P_i$ (respectively, $P_{n+2}$) are equal to $D_i$ (respectively, $D_{n+2})$, there is one occurrence of $D_{i-1}$ (respectively, $D_n$) between the two $D_i$ (respectively, $D_{n+2}$), and there is no  occurrence of $D_{i+1}$ between them. Hence any other simple diagram appearing between these two occurrences of $D_i$ (respectively, $D_{n+2}$) in the factorization of $D_{\bf w}$ commutes with them. Therefore there exists a reduced expression ${\bf w'}$ in the same commutation class of ${\bf w}$ in which the factor $s_is_{i-1}s_i$ (respectively, $s_{n+2}s_ns_{n+2}$) appears and so $w$ is not FC. A similar argument holds for the patterns of the form $Q_j$, $j=0,\ldots,n$.
\end{proof} 

\begin{Corollary}\label{rem:loops}
Let $w\in \FC(\Dn)$, then
\begin{itemize}
\item[(1)] $D_w$ contains a $\lob$ (respectively a $\loc$) if and only if there exists a reduced expression of $w$ containing a factor $s_0s_1$ (respectively $s_{n+1}s_{n+2}$).
\item[(2)] $D_w$ contains a $\lob$ and the vertical edge $\{1,1'\}$ (respectively a $\loc$ and the vertical edge $\{n+2, (n+2)'\}$) if and only if there exists a reduced expression of $w$ containing a factor $s_2s_0s_1s_2$ (respectively $s_n s_{n+1}s_{n+2}s_n$).
\end{itemize}    
\end{Corollary}

\begin{proof} We only prove the direct implications since the reverses are clear.

\noindent (1) Consider any concrete simple representation $\mathcal{C}(D_{\bf w})$. The only configurations giving a loop decorated with a $\bullet$ are those depicted in Figure \ref{fig:loop_pattern} or those obtained from them by displacing the $\bullet$ inside the highlighted boxes, here $i \leq n$. The first case is impossible by Lemma \ref{lemma:pattern}. From the other it is easy to derive a reduced expression containing the factor $s_0s_1$, by exchanging the positions of the diagrams between $D_0$ and $D_1$. An analogous argument works for $\loc$.

\begin{figure}
    \centering
    \includegraphics[scale=0.5]{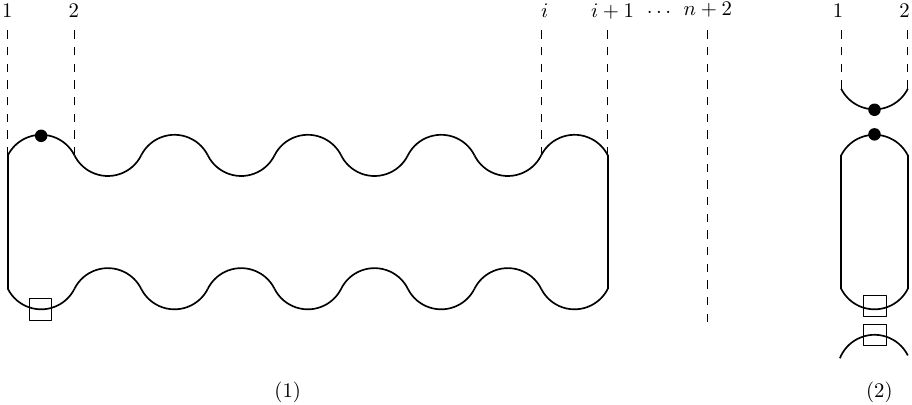}
    \caption{}
    \label{fig:loop_pattern}
\end{figure}

\noindent (2) By the previous point, there exists a reduced expression ${\bf w}$ such that $D_{\bf w}$ contains the factor $D_0D_1$. The associated $\mathcal{C}(D_{\bf w})$ has a path from node $1$ to node $1'$ if and only if there are two occurrences of $D_2$, one before and one after $D_0D_1$, otherwise, for topological reasons, in $D_{\bf w}$ there would appear a pattern of type $(P_i)$. An analogous argument works for $\loc$ and for the edge $\{n+2, (n+2)'\}$.
\end{proof}

\begin{Proposition}\label{prop:northface}
    Let $w\in \FC(\Dn)$, then 
    \begin{itemize}
        \item[(1)] $D_w$ has a simple edge $e=\{i, i+1\}$ of the form $(\smile)$ on the north face if and only if there exists a reduced expression of $w$ starting with $s_i$;
        \item[(2)] $D_w$ has a simple edge of the form $(\stackrel{\bullet}{\smile})$ (respectively, $(\stackrel{\circ}{\smile})$) on the north face if and only if there exists a reduced expression of $w$ starting with $s_0$ (respectively, $s_{n+2}$) and no reduced expression of $w$ starts with $s_1$ (respectively, $s_{n+1}$).    
    \end{itemize}
\end{Proposition}

\begin{proof}

 Let $\mathbf{w}\in \mathcal{R}(w)$ and consider the concrete simple representation $\mathcal{C}(D_{\bf w})$. The necessary conditions for both (1) and (2) are clear. Now we show the other implications. 
 
 For (1), an undecorated closed path from node $i$ to node $i+1$ can be obtained either from the stacking of the simple diagrams as in Figure \ref{fig:FCrepres}(4), or in a way that the path starting from $i$ changes direction at least once from left to right or vice versa, reaching $i+1$. The latter case is not possible since the path would contain one of the forbidden patterns in Figure \ref{fig:patterns}: three examples are shown in Figure \ref{fig:FCrepres}(1)-(3), where diagrams that commute are depicted at the same height. Therefore the only possibility is the first one in which $D_i$ commutes with all the simple diagrams above it. Hence we can find a reduced expression of $w$ starting with $s_i$. 
 
 For (2), if $e=\{1,2\}$ (respectively, $\{n+1,n+2\}$) is of the form $(\stackrel{\bullet}{\smile})$ (respectively, $(\stackrel{\circ}{\smile})$), the same argument as above implies that it can only be obtained from the stacking of the simple diagrams as in Figure \ref{fig:FCrepres}(4), where the highlighted box is filled with a $\bullet$ (respectively, $\circ$). Moreover, in $D_w$ there cannot be any $\lob$ (respectively, $\loc$), otherwise by the relation (r3) (respectively, (r4)) in Figure \ref{rel}, the simple edge would not be decorated. By Corollary \ref{rem:loops}, no reduced expression of $w$ can start with $s_1$ (respectively, $s_{n+1}$).

\begin{figure}[h]
    \centering
    \includegraphics[scale=0.6]{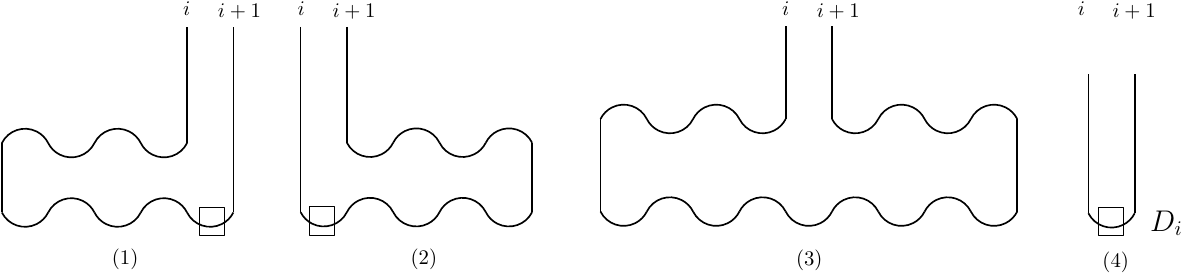}
    \caption{}
    \label{fig:FCrepres}
\end{figure}
    
\end{proof}

Recall the definitions of the (PZZ) heaps $H_{Z_1}$ and $H_{Z_2}$ given in Remark \ref{rem:pzz-heaps}, and the definitions of the PZZ-diagrams $Z_1$ and $Z_2$ given in Figure \ref{degeneratePZZ}.

\begin{Corollary} \label{cor:degeneratePZZ}
Let $w\in \FC(\Dn)$ and $H(w)$ be its FC heap, then
\begin{itemize}
    \item[(1)] $H(w)=H_{Z_1}$ if and only if $D_w=Z_1$;
    \item[(2)] $H(w)=H_{Z_2}$ if and only if $D_w=Z_2$.
\end{itemize}

\end{Corollary}

\begin{proof}
 We only prove (1) since (2) follows from (1) and $Z_2=F(Z_1)$. 
 
   If $H(w)=H_{Z_1}$, then a reduced expression of $w$ is $s_0s_1s_2\cdots s_{n+1}s_{n+2}$, from which $D_w=D_0D_1D_2\cdots D_{n+1}D_{n+2}=Z_1$.
   
 Now suppose $w\in \FC(\Dn)$, with $D_{w}=Z_1$. Since one $\lob$ and one $\loc$ occur in $Z_1$, by Corollary \ref{rem:loops} there exists a reduced expression containing both factors $s_0s_1$ and $s_{n+1}s_{n+2}$ once and, since $\lob$ is higher than $\loc$, then $s_0s_1$ is on the left of $s_{n+1}s_{n+2}$. We consider a reduced expression
 ${\bf w}=a s_0s_1 b s_{n+1}s_{n+2}c$ with $a$ and $c$ factors of minimal length.
    By minimality, $a=s_k s_{k-1}\cdots s_3 s_2$, and it is easy to check that if $a\neq \emptyset$ then the unique non-propagating edge on the north face of $D_{w}$ is different from $\{1,2\}$, against the hypothesis. Hence $a=\emptyset$ and similarly also $c = \emptyset$, so ${\bf w}=s_0s_1bs_{n+1}s_{n+2}$.

Now we want to show that $b=s_2s_3\cdots s_n$. Let $b=s_jb'$. If $j\neq 2$ there exists a reduced expression of $w$ starting with $s_j$ and so, by Proposition \ref{prop:northface}, $\{j, j+1\}$ is a non-propagating edge on the north face of $Z_1$, against the hypothesis. By iterating this argument one can show that $b=s_2s_3\cdots s_{k-1}s_k$, and $k$ has to be equal to $n$, otherwise $\{n+1, n+2\}$ is a non-propagating edge on the north face of $Z_1$. Therefore, ${\bf w}=s_0s_1s_2\cdots s_n s_{n+1}s_{n+2}$ and so $H(w)=H_{Z_1}$.

   \end{proof}

\begin{Lemma}\label{Pheaps}
Let $w\in \FC(\Dn)$. If $w\in$ $(\LP)$ (respectively $(\RP)$, $(\LRP)$) then $D_w$ is a P-diagram of type LP (respectively RP, LRP). 
\end{Lemma}

\begin{proof}
Assume that $w\in$ (LP). By Definition \ref{family}, a reduced expression $\mathbf{w}$ of $w$ contains the factor $s_{j_{\ell}}s_{j_{\ell} -1}\cdots s_2s_0s_1s_2\cdots s_{j_{\ell} -1} s_{j_{\ell}}$. The corresponding product $$D_{j_{\ell}}D_{j_{\ell} -1}\cdots D_2D_0D_1D_2\cdots D_{j_{\ell} -1} D_{j_{\ell}}$$ is a diagram with the leftmost $j_{\ell} -1$ edges vertical and undecorated and with a unique $\lob$. Since there are no other occurrences of $s_r$ with $r\leq j_l$ in $\mathbf{w}$, then such vertical edges appear in $D_w$ too. Furthermore, if $s_{j_{\ell}+1}$ occurs once in $\mathbf{w}$ then $D_w$ contains either the edge $\{j_{\ell}, j_{\ell} +1\}$ or  $\{j'_{\ell}, (j_{\ell} +1)'\}$. If $\mathbf{w}$ contains two occurrences of $s_{j_{\ell}+1}$ then, by LP definition, there must be also an occurrence of $s_{j_{\ell}+2}$ between them, then it can be easily proved that either at least one of the edges leaving $j_{\ell}$ and $j'_{\ell}$ is a non-propagating edge or $D_w$ has the vertical edge $\{j_{\ell}, j'_{\ell}\}$ decorated with a $\circ$.
Hence, by Definition \ref{defP}, $D_w$ is a P-diagram of type LP. Analogous arguments can be applied for the two other cases.
\end{proof}

\begin{Lemma}\label{ALTheaps}
Let $w\in \FC(\Dn)$. If $w\in (\ALT)$ then $D_w$ is an ALT-diagram.
\end{Lemma}

\begin{proof}
Suppose that $D_w$ is a P-diagram, then by Corollary \ref{rem:loops} there exists a reduced expression of $w$ that contains at least one of the factors $s_2s_0s_1s_2$ and $s_n s_{n+1}s_{n+2} s_n$ that is not possible since $w\in$ (ALT). 

 Now suppose $D_w$ is a PZZ-diagram. By the argument right above for the P-diagrams, $D_w$ cannot contain the vertical edges $\{1,1'\}$ and $\{n+2, (n+2)'\}$. 
 Hence the non-propagating edges of $D_w$ are 
either $\{1, 2\}$ and $\{(n+1)', (n+2)'\}$ or $\{n+1, n+2\}$ and $\{1', 2'\}$. If $\mathtt{l}\ge 2$ (respectively $\mathtt{r}\ge 2$), by Corollary~\ref{rem:loops} a reduced expression of $w$ has a factor $b=s_0s_1 a_1 s_{n+1}s_{n+2} a_2 s_0s_1$ (respectively $b'=s_{n+1}s_{n+2} a_1 s_{0}s_{1} a_2 s_{n+1}s_{n+2}$), and by the argument in the proof of Corollary~\ref{cor:degeneratePZZ}, $b$ (respectively $b'$) must contain the factor $s_ns_{n+1}s_{n+2}s_n$ (respectively $s_2s_0s_1s_2$). Both cases cannot occur since $w\in (\ALT)$. 

Finally, suppose $\mathtt{l}=\mathtt{r}=1$. By Corollary \ref{cor:degeneratePZZ}, $D_w\neq Z_1, Z_2$, since by Definition \ref{family}, $H_{Z_1},H_{Z_2}\notin (\ALT)$, hence there are two cases left. If $\{1, 2\}$ (respectively $\{(n+1)', (n+2)'\}$) is the non-propagating edge then $D_w$ is of type $\langle \circ \, \bullet \rangle$ (respectively $\langle \bullet \, \circ \rangle$) and $w$ must have both factors $s_2s_0s_1s_2$ and $s_ns_{n+1}s_{n+2}s_n$ which are not allowed. 

Therefore $D_w$ has to be an ALT-diagram.
 
\end{proof}

Now we prove that the length of $D_w$ as defined in Definition \ref{deflength} is consistent with the length of $w$.
 
 Denote by $\mu_k(e)$ the number of intersections of $e$ with the vertical line $k+1/2$ in $\mathcal{C}(D_{\mathbf{w}})$ and set 
  \begin{equation}
      \mu_k(\mathcal{C}(D_{\mathbf{w}})):=\sum_{e} \mu_k(e) \quad \mbox{and} \quad \mu(e):=\sum_{k=1}^{n+1} \mu_k(e),
  \end{equation}
   where the first sum runs over all the edges of $\mathcal{C}(D_{\mathbf{w}})$. Now, given two different reduced expressions $\mathbf{w}$ and $\mathbf{w'}$ of $w$ we have $\mathcal{C}(D_{\mathbf{w}})\neq \mathcal{C}(D_{\mathbf{w'}})$, while $\mu_k(\mathcal{C}(D_{\mathbf{w}}))=\mu_k(\mathcal{C}(D_{\mathbf{w'}}))$ for any $k$. Moreover, 
   \begin{equation}\label{eq:mu-length}
       \sum_{k=1}^{n+1}\mu_k(\mathcal{C}(D_{\mathbf{w}}))=\sum_e \mu(e) = 2\ell(w).
   \end{equation}
   Now we are ready to prove the following result.

\begin{Proposition}\label{fclength}
Let $w\in \FC(\Dn)$, then $\ell(w)=\ell(D_w)$.
\end{Proposition}

\begin{proof}
 First assume that $D_w$ is a PZZ-diagram. Then by Lemma \ref{PZZheaps}, $w\in (\PZZ)$. As in the proof of Lemma \ref{PZZheaps},  we can suppose that  $$\mathbf{w}=s_is_{i+1}\cdots s_n (s_{n+1}s_{n+2}s_n\cdots s_2s_0s_1s_2\cdots s_n)^h s_{n+1}s_{n+2}s_n \cdots s_j,$$
 is a reduced expression of $w$, with $1\le i,j<n+1$. We have that  $$\ell(w)= 2n-i-j+4+(2n+2)h=3-i+n-j+(2h+1)(n+1)=\ell(D_w)$$ since $\mathtt{l}=h$ and $\mathtt{r}=h+1$ (Definition \ref{deflength}). 

Now assume $D_w$ is an ALT-diagram and consider a decorated edge $e$ in $\mathcal{C}(D_{\bf w})$. We now show that 
\begin{equation}\label{eqn:mu}
    \frac{1}{2}(\mu(e)-j+i)= \weight(e),
\end{equation}
 where $e=\{i^*, j^*\}$ and $i\leq j$.

If $e$ is a non-propagating edge $\{i, j\}$ on the north face decorated with both $\bullet$ and $\circ$ (see Figure \ref{ALTedge}(A)), then $\weight(e)=n+1 -j+i$.  Furthermore,  $\frac{1}{2}(\mu(e)-j+i)=\frac{1}{2}(\sum_{k=1}^{i-1}\mu_k(e) + \sum_{k=j}^{n+1}\mu_k(e))= i-1 + n+1 -j +1= n+1-j+i= \weight(e)$. Now if $e$ is a propagating edge $\{i, j'\}$ with $i<j$ and is decorated with a $\bullet$ followed by $(\circ \, \bullet)^k$, where $k\in \mathbb{N}_0$ (see Figure \ref{ALTedge}(B)), then, one can verify that $\frac{1}{2}(\mu(e) - j+i)=i-1 + k(n+1)=  \weight(e)$. Similar computations show that \eqref{eqn:mu} holds if $e$ is decorated in a different way. Moreover, when $e=\{i, j'\}$ with $j<i$, \eqref{eqn:mu} holds with $i$ and $j$ exchanged.

\begin{figure}[hbtp]
\centering
\includegraphics[scale=0.6]{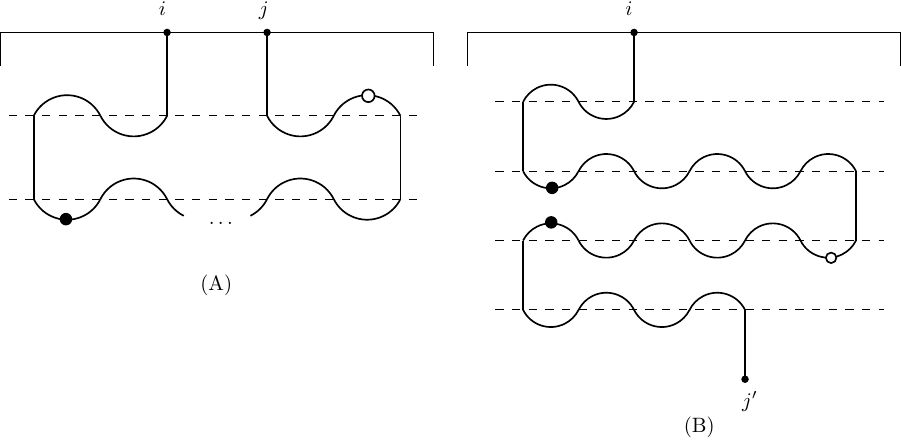}
\caption{Edges in concrete simple representations.}
\label{ALTedge}
\end{figure}

A quick check shows that $\sum_e \mu(e)=2\sum_e \weight(e) + \sum_{h=1}^{n+1}\nu_h(D_w)$, where the sum on the left hand side runs over all the edges of $\mathcal{C}(D_{\bf w})$, while the first summand on the right hand side runs over all the edges of $D_w$, loop edges included. From Definition~\ref{deflength}, Definition~\ref{rem:simplerepr} and \eqref{eq:mu-length}, it follows that $\ell(D_w)= \frac{1}{2} \sum_e \mu(e)=\ell(w)$. 

Finally, assume that $D_w$ is a LP-diagram. 
Then a direct computation shows that $\sum_e \mu(e)$ is equal to $4 j_\ell-2$ plus the number of intersections between all the vertical lines and the edges starting and ending on the right of $\{j_\ell -1, (j_\ell -1)'\}$ in $\mathcal{C}(D_{\bf w})$, namely the part of $\mathcal{C}(D_{\bf w})$ corresponding to $\widehat{D_w}$. Since $D_{\bf w}$ is an ALT-diagram, this number is equal to $\ell(\widehat{D_w})$. This gives $\ell(w)=\frac{1}{2}\sum_e \mu(e)=2j_\ell -1+\ell(\widehat{D_w})=\ell(D_w)$ as in Definition~\ref{deflength}(d). A similar computation can be done for the other two cases RP and LRP, and this concludes the proof.
\end{proof}

In Figure \ref{simplerepres}, $4 j_\ell-2 = 10$ counts the intersections of the vertical lines with the black and blue edges of $D_{\bf w}$ while there are 8 intersections with the red edges, those starting and ending on the right of $\{2,2'\}$. The total number of intersections is $18=2\ell(w)$ as expected.

\begin{Proposition}\label{factorization}
Every admissible diagram $D$ is of the form $D_w$ where $w\in \FC(\Dn)$.
\end{Proposition}

\begin{proof}
If $D$ is a PZZ or a P-diagram with $\ab(D)=1$, the thesis follows from Lemma \ref{ad1}, observing that any factorization appearing in the proof of the lemma is associated to a FC element $w$, as classified in Definition \ref{family}. Hence $D=D_w$.

Assume that $D$ is an ALT-diagram or a P-diagram with $\ab(D)>1$. We proceed by induction on $\ell(D)$. If $\ell(D)=1$, then it easy to see that it must be $D=D_i$ for some $i\in\{0, 1 , \ldots, n+2\}$, and trivially $s_i\in \FC(\Dn)$.
Suppose the thesis holds for all diagrams of length $k$ and let $D$ be of length $k+1$. By Lemma \ref{unicity} and Lemma \ref{unicityP}, there exist an index $i\in \{0, 1 , \ldots, n+2\}$ and a unique admissible diagram $D'$ such that $D=D_iD'$ and $\ell(D')=k$. Hence $D'=D_{w'}$ with $w'\in \FC(\Dn)$ and $\ell(w')=k$. By letting $w:= s_i w'$, we claim that $w\in\FC(\Dn)$ with $\ell(w)=\ell(w')+1$, from which $D=D_w$ and the result follows. We now prove the claim. 

If $\ell(w)=\ell(w')-1$ then there exists a reduced expression of $w'$ starting with a $s_i$. Hence 
$D_iD_{w'}$ contains a $\delta$ factor, but $D$ admissible implies $D$ irreducible and so it cannot contain undecorated loops, hence $\ell(w)=\ell(w')+1$. If $w \not\in \FC(\Dn)$ then by Proposition~\ref{prop:caracterisation_fullycom} there exists a reduced expression of $w$ containing a braid relation. After applying all the possible relations between the simple diagrams listed after Figure \ref{simple}, we have $$D=D_iD_{w'}=\delta^p D_{w''}$$ with $w'' \in \FC(\Dn)$ and $\ell(w'')<\ell(w')+1$. If $p\neq 0$ we have a contradiction since $D$ is admissible. If $D=D_{w''}$ then $\ell(D)<k+1$ which is also in contrast with the hypothesis.

\end{proof}

\begin{Theorem}\label{thalgebra}
The $\Zd$-algebras $\mathcal{M}[\dbD]$ and $\DD$ are equal. Moreover, the set of admissible diagrams is a basis for $\DD$.
\end{Theorem}

\begin{proof}
We know that $\{b_w \mid w \in \FC(\Dn)\}$ is a basis of $\TL(\Dn)$, hence by Proposition~\ref{factorization} and the fact that the image of any $b_w$ is an admissible diagram (by Lemmas \ref{PZZheaps}, \ref{Pheaps} and \ref{ALTheaps})  it follows that $\mathcal{M}[\dbD]$ is a $\Zd$-subalgebra of $\PLRn$. Since the simple diagrams are in $\mathcal{M}[\dbD]$ and $\DD$ is the smallest algebra containing them, then the two algebras are equal. By Proposition \ref{prop:basis}, the set of admissible diagrams is a basis for $\mathcal{M}[\dbD]$, then the statement is proved.
\end{proof}

We can finally prove our main result.

\begin{Theorem}\label{injec}
The map $\tilde{\theta}_D: \TLD \rightarrow \DD$ is an algebra isomorphism. 
\end{Theorem}

\begin{proof}
It is enough to show that $\tilde{\theta}_D$ is injective. Let $w_1, w_2$ be two distinct elements in $\FC(\Dn)$. By definition, $\tilde{\theta}_D(b_{w_1})=D_{w_1}$ and $\tilde{\theta}_D(b_{w_2})=D_{w_2}$. By Lemmas~\ref{PZZheaps}, \ref{Pheaps} and \ref{ALTheaps}, if $w_1$ and $w_2$ belong to two different families of heaps then $D_{w_1}\neq D_{w_2}$. 

If both $w_1$ and $w_2$ belong to (PZZ), it is easy to verify that $w_1\neq w_2$ implies $D_{w_1}\neq D_{w_2}$. In fact, either the endpoints of $H(w_1)$ and $H(w_2)$ are different, or they coincide but differ in the number or relative order of occurrences of $s_0s_1$ and $s_{n+1}s_{n+2}$. Hence, $D_{w_1}$ and $D_{w_2}$ either have different non-propagating edges, or differ in the number or relative vertical position of $\lob$ and $\loc$, respectively.
 
Now, suppose by contradiction that $\tilde{\theta}_B$ is not injective and let $w_1\neq w_2$ be such that $D_{w_1}=D_{w_2}=D$ with $D$ of minimal length. First note that by Proposition~\ref{fclength}, $\ell(w_1)=\ell(w_2)$ and, from above, 
that $w_1$ and $w_2$ are both in (ALT) or in (P).
By Propositions~\ref{unicity} and \ref{unicityP}, there exist an index $i\in \{0, 1, \ldots, n+2\}$ and a unique admissible diagram $D'$ such that $D_{w_1}=D_iD'$, with $\ell(D')=\ell(D)-1$. This factorization implies that $D_{w_1}$ has a simple edge of the form $(\stackrel{\bullet}{\smile})$ if $i=0$, $(\stackrel{\circ}{\smile})$ if $i=n+2$, and $(\smile)$ if $1\le i \le n+1$.  By Proposition~\ref{prop:northface}, it follows that there exists a reduced expression of $w_1$ starting with $s_i$. Hence $w_1=s_i a$ where $a\in \FC(\Dn)$, $\ell(a)=\ell(w_1)-1$ and $$D_{w_1}=D_i D_a,$$ with $D_a$ admissible. Then, $D_a=D'$ by the uniqueness of $D'$.
By the same argument   
$$D_{w_2}=D_i D_b,$$
with $w_2=s_ib$, $b \in \FC(\Dn)$, $\ell(b)=\ell(w_2)-1$ and $D_b=D'$.
Hence $D_a=D_b=D'$, and by the minimality of $\ell(D)$, we have that $a=b$ from which $w_1=w_2$.
\end{proof}

\begin{Example}
Consider the heaps depicted in Figure~\ref{heaps}. The image of the (ALT) element via $\tilde{\theta}_D$ is the diagram in Figure~\ref{ALTdiag}, the image of the (PZZ) element is in Figure~\ref{PZZdiag}, the image of the (LRP) element is in Figure~\ref{LRPdiag}.
\end{Example}

\section{Type $\widetilde{B}$}\label{Green}

In this section we extend to type $\Bn$ the previous results for type $\Dn$. In particular, we describe a new algebra of decorated diagrams called admissible, denoted by $\DB$, that turns out to be isomorphic to $\TL(\Bn)$. We list the relevant definitions and we state the main theorems without giving formal proofs, since they are similar to those in type $\widetilde{D}$. Following a suggestion of Richard Green~\cite{RG}, the injectivity of the map $\tilde{\theta}_B$
will be derived from our main result, Theorem~\ref{injec}.

From now on, by abuse of notation we will use same the symbols for elements in $\Bn$ and $\Dn$, in $\TL(\Bn)$ and $\TL(\Dn)$, and for diagrams in $\DB$ and $\DD$.
\medskip

To define admissible diagrams in type $\Bn$ we consider the algebra $\PLRb$ of decorated diagrams introduced in Section~\ref{sec:undecorated}, where the decorations are in the set $\Omega'=\{\bullet, \circ, \tb\}$, with $L=\{\bullet\}$ and $R=\{\circ, \tb\}$. By introducing the following set of reductions, 
\begin{figure}[hbtp]
\centering
\includegraphics[scale=0.6]{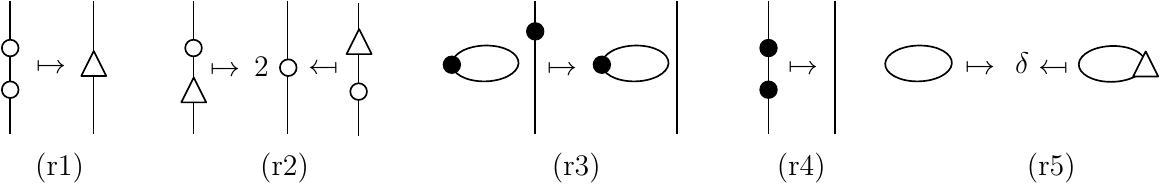}
\caption{The reductions of $\PLRk$.}
\label{relB}
\end{figure}
one can show, as in Section~\ref{sec:undecorated}, that the quotient algebra $\PLRkb$ of $\PLRb$ modulo the relations determined by the reductions in Figure~\ref{relB} is a $\Zd$-algebra having as a basis the set of irreducible decorated diagrams with decorations in $\Omega'$.

We can now state our main definition. 

\begin{Definition}\label{definition-admissible}
An irreducible diagram $D\in T_{n+2}^{LR}(\Omega')$ is called $\tilde{B}$-\textit{admissible} if it satisfies the following conditions:

\begin{itemize}

\item[(B1)] In $D$ loops are all of the same type, either $\lob$ or $\lot$. Moreover, if $\ab(D)>1$ then there might be at most one occurrence of $\lob$. 
    \begin{figure}[h]
	\centering
	\includegraphics[width=0.3\linewidth]{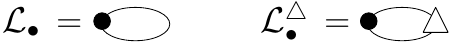}
	\caption{Allowable loops in $\tilde{B}$-admissible diagrams.}
	\label{loop-B}
\end{figure}

\end{itemize}

Assume $\ab(D)>1$.

\begin{itemize}
\item[(B2)] The total number of $\lot$ and of $\bullet$ on non-loop edges of $D$  must be even and the total number of $\circ$ is either zero or two.
\smallskip
\item[(B3)] If $D$ has no propagating edges, a $\circ$ must appear as last decoration on the non-propagating edges connected to the nodes $n+2$ and $(n+2)'$ respectively.
\smallskip
\item[(B4)] If $\textit{e}=\{1,1'\}$, then exactly one of the following three conditions holds:
\begin{itemize}
\item[(a)] \textit{e} is undecorated;
\item[(b)] \textit{e} is decorated by a single $\tb$;
\item[(c)] \textit{e} is decorated by an alternating sequence of $\bullet$ and $\tb$ and there are no loops.
\end{itemize}
\smallskip

\item[(B4)$'$]  If $e=\{n+2, (n+2)'\}$, then only one of the following conditions holds:
\begin{itemize}
\item[(a)$'$] $e$ is undecorated;
\item[(b)$'$] $e$ is decorated by a single $\tb$;
\item[(c)$'$] $e$ is decorated by an alternating sequence of $\bullet$ and $R$-decorations such that the first and the last decorations are $\circ$.
\end{itemize}
Moreover, if $e\neq \{n+2, (n+2)'\}$ is the edge  connected to $n+2$ (respectively, $(n+2)'$), then a unique $\circ$ occurs on $e$.
\end{itemize}

\smallskip
Assume $\ab(D)=1$.
\begin{itemize}
\item[(B5)] The western side of $D$ is equal to one of those depicted in Figure  \ref{west}, as required for type $\Dn$, see Definition~\ref{definition-admissible} (A4).
\smallskip

\noindent The eastern side of $D$ is equal to one of those depicted in Figure~\ref{est}, where ${\bf d}$ represents a sequence of blocks (possibly empty) such that each block is a single $\tb$. The $\circ$ decoration occurring on the propagating edge on the north (respectively, south) face has the highest (respectively, lowest) relative vertical position of all decorations on any propagating edge and $\lob$. The situation in Figure~\ref{est}(B) can only occur if no $\bullet$ is in $D$.
\end{itemize}

\begin{figure}[hbtp]
	\centering		
		\includegraphics[height=25mm, keepaspectratio]{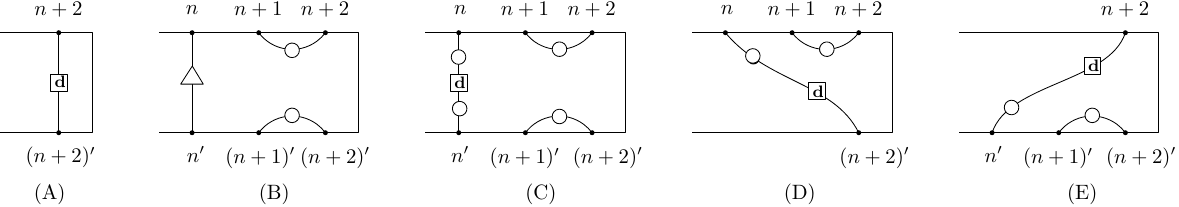}
		\caption{Eastern side of an admissible diagram with $\ab(D)=1$.}
	\label{est}
\end{figure}
\end{Definition}

As in Section~\ref{sec:undecorated}, one can show that the set $\db$ of admissible diagrams  spans the $\Zd$-submodule $\mathcal{M}[\db]$, which turns out to be a $\Zd$-subalgebra of $\PLRnb$. Moreover as in type $\Dn$, the set $\db$ splits in three nonempty disjoint families: the PZZ, the P and the ALT-diagrams.   

\begin{Definition} Let $D\in \db$.
\begin{itemize}
    \item[(PZZ)]  We say that $D$ is a \textit{PZZ-diagram} if $\ab(D)=1$ and $\mathtt{l},\mathtt{t}\geq 1$, where $\mathtt{l}$ denotes the number of $\lob$ and $\mathtt{t}$ the number of $\tb$ in $D$.
    \item[(P)] $D$ is a \textit{P-diagram} if it satisfies Definition~\ref{defP} where any $\loc$ is replaced by a $\tb$ on the vertical edge $\{n+2,(n+2)'\}$.
    \item[(ALT)] $D$ is an \textit{ALT-diagram} if $D$ is not a P-diagram or a PZZ-diagram.
\end{itemize}
\end{Definition}

To avoid further notation, we denote again by $D_0,\ldots,D_n, D_{n+1}$ the \textit{simple diagrams of type $\Bn$}, defined as follows. The first $n+1$ diagrams are equal to $D_0, \ldots, D_{n}$ of type $\Dn$, while $D_{n+1}$ is equal to $D_{n+2}$ of type $\Dn$, see Figure~\ref{simple}. 
\smallskip

Let $\DB$ be the $\Zd$-subalgebra of $\PLRnb$ generated as a unital algebra by the simple diagrams of type $\Bn$, then one can show the following result.

\begin{Theorem}\label{thalgebraB}
The $\Zd$-algebras $\mathcal{M}[\db]$ and $\DB$ are equal. Moreover, the set of admissible diagrams is a basis for $\DB$.
\end{Theorem}

Consider the $\Zd$-algebra homomorphism $\tilde{\theta}_B : \TL(\Bn) \rightarrow \DB$ defined as usual by $b_i \mapsto D_i$ for all $i=0, \ldots, n+1$. We will show that the injectivity of $\tilde{\theta}_B$ can be deduced from that of $\tilde{\theta}_D$ using the following maps.

\begin{Definition}  
Let  $w \in \FC(\Bn)$, the map $\phi:\FC(\Bn) \rightarrow \FC(\Dn)$ is defined as follows: 
\begin{itemize}
\item[a)] If no reduced expression of $w$ contains the factor $s_{n+1}s_{n}s_{n+1}$, consider any $\mathbf{w}=s_{i_1}\cdots s_{i_k} \in \mathcal{R}(w)$, then $\phi(w)$ is the element with a reduced expression obtained by replacing each $s_{n+1}$ in ${\bf w}$ by $s_{n+1} s_{n+2}$.

\item[b)] If there exists a reduced expression ${\bf w}$ containing a factor $s_{n+1}s_ns_{n+1}$, then $\phi(w)$ is the element with a reduced expression obtained by replacing in $\bf w$ the second occurrence of $s_{n+1}$ by $s_{n+2}$ in any of these factors. 
\end{itemize}

\end{Definition}

\begin{Definition}
    Let  $f_{\phi}:\TL(\Bn)\rightarrow\TL(\Dn)$ be the map defined by $f_{\phi}(b_w)=b_{\phi(w)}$  for any $w\in \FC(\Bn)$. 
\end{Definition}

\begin{Lemma}\label{rem:shape}\
\begin{enumerate}
\item The maps $\phi$ and $f_\phi$ are injective.   
\item If $w\in \FC(\Bn)$ then $sh(D_{\phi(w)})=sh(D_w)$.
\end{enumerate}
\end{Lemma}

\begin{proof}  
 The injectivity of $\phi$ follows directly from its definition, while that of $f_\phi$ from that of $\phi$ since $\{b_w \mid w \in \FC(\Bn)\}$ is a basis. Recall that in $sh(D_w)$ loops are not considered, so (2) comes from the definition of $\phi$ since the first $n+1$ simple diagrams of type $\Bn$ and $\Dn$ are all equal and $sh(D_{n+1})=sh(D_{n+2})$. 
\end{proof}

\begin{Definition}\label{def:g-function}
    Let  $g: \DB \rightarrow \DD $ be the map defined as follows. If $D\in \DB$, then $g(D)$ has the same shape of $D$, the same $L$-decorations of $D$ and the following $R$-decorations:
    \begin{itemize}
        \item[(a)] none, if $D$ has no $R$-decorations;
        \item[(b)] a $\loc$, if $D$ has only two $\circ$ decorations;
        \item[(c)] a $\circ$ in the same position of each $\tb$ in $D$, if $D$ is an ALT or LP-diagram containing an even number of $\tb$;
        \item[(d)] a $\circ$ in the same position of each $\tb$ in $D$ and one $\circ$ on the edge leaving $(n+2)'$ with lowest vertical position, if $D$ is an ALT or LP-diagram containing an odd number of $\tb$;
        \item[(e)] a $\loc$ in the same vertical position of the $\tb$ in $D$, if $D$ is a RP or LRP-diagram;
        \item[(f)] a $\loc$ in the same vertical position of each $\tb$ and an extra $\loc$ with highest (respectively lowest) vertical position if a $\circ$ occurs on the edge leaving $n+2$ (respectively $(n+2)'$) of $D$, if $D$ is a PZZ-diagram.
    \end{itemize}    
\end{Definition}

\begin{figure}[hbtp]
	\centering
    \includegraphics[scale=0.7]{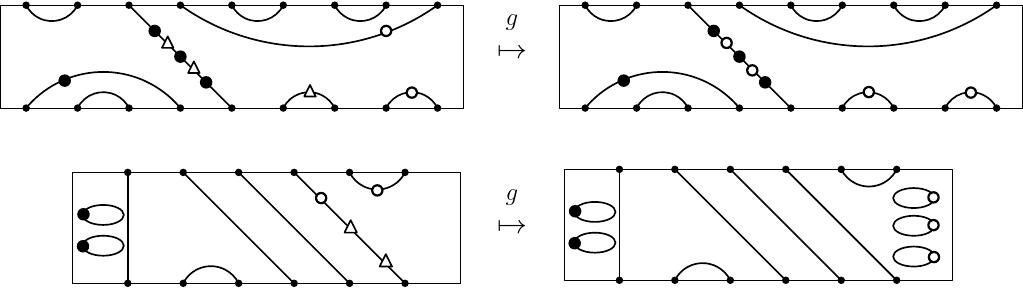} 
	\caption{Examples of the map $g$, cases (d) and (f) respectively.}
	\label{g-function}
\end{figure}

\begin{Theorem}
    The following diagram 
\begin{center}
    \begin{tikzcd} 
\TL(\Bn) \arrow{r}{\tilde{\theta}_B} \arrow{d}{f_\phi} &\DB \arrow{d}{g}\\ 
\TL(\Dn) \arrow{r}{\tilde{\theta}_D} & \DD 
\end{tikzcd}
 \end{center}
 commutes. Furthermore the map $\tilde{\theta}_B$ is an isomorphism.
\end{Theorem}

\begin{proof}
Let $w\in \FC(\Bn)$, by Lemma~\ref{rem:shape}, $D_w$ in $\DB$ and $D_{\phi(w)}$ in $\DD$ have the same shape. They also have exactly the same number of $\bullet$ decorations placed in the same positions. The differences between the two diagrams concern the loops and the $R$-decorations. More precisely, if an edge in $D_w$ contains a $\tb$, then the corresponding edge in $D_{\phi(w)}$ contains a $\circ$ in the same position. If the edges leaving $n+2$ and $(n+2)'$ contain a $\circ$ in $D_w$, then the corresponding edges in $D_{\phi(w)}$ may contain a $\circ$ or not in the same position, depending on the parity of the number of $\tb$. We now detail all the different cases.

    \begin{itemize}
        \item[(1)] If no $s_{n+1}$ occurs in $w$ then $\phi(w)=w$ in $\Dn$, then $D_{\phi(w)}$ has no $\circ$ decorations since no $D_{n+2}$ occurs. On the other hand, also $\tilde{\theta}_B(b_w)=D_w$ has no $R$-decorations, so $g(D_w)=D_w=D_{\phi(w)}$.

        \item[(2)] If a unique $s_{n+1}$ occurs in $w$, then $D_{n+1}D_{n+2}$ is a factor of $D_{\phi(w)}$ then it is easy to check that $D_{\phi(w)}$ has a single $\loc$ and no other $\circ$ decorations. On the other hand, $\tilde{\theta}_B(b_w)=D_w$ has two $\circ$ decorations, so, by Definition~\ref{def:g-function}(b), $g(D_w)=D_{\phi(w)}$.

        \item[(3)] If $w$ is in (ALT) or (LP) and it has $2k$ occurrences of $s_{n+1}$, then in any factorization of $D_{\phi(w)}$ there are $k$ occurrences of $D_{n+1}$ and $D_{n+2}$ that alternate. In particular, the nember of $\circ$ decorations is $2k$ and one of them is the last in the edge leaving $(n+2)'$. Note also that the first decoration in the edge leaving $n+2$ is not a $\circ$. On the other hand, in $D_w$ the first (respectively last) decoration  in the edge leaving $n+2$ (respectively $(n+2)'$) is a $\circ$ and the total number of $\tb$ decoration is $2k-1$. Hence, by Definition~\ref{def:g-function}(d), $g(D_w)=D_{\phi(w)}$.

        \item[(4)] If $w$ is in (ALT) or (LP) and it has $2k+1$ occurrences of $s_{n+1}$ then in $D_{\phi(w)}$ there are $k$ occurrences of $D_{n+2}$ and $k+1$ of $D_{n+1}$ that alternate. Again, in $D_{\phi(w)}$ the number of $\circ$ decorations is $2k$. On the other hand, $D_w$ has $2k$ occurrences of $\tb$ decorations. Hence, by Definition~\ref{def:g-function}(c), $g(D_w)=D_{\phi(w)}$. 

        \item[(5)] If $w\in (\RP)$ or $(\LRP)$, then the factor $D_{j_r}\cdots D_{n}D_{n+1}D_{n+2}D_n \cdots D_{j_r}$ appears in a factorization of $D_{\phi(w)}$ giving rise to a $\loc$ and to the undecorated vertical edge $\{n+2, (n+2)'\}$, see Corollary~\ref{rem:loops}. On the other hand, $D_w$ has the vertical edge $\{n+2, (n+2)'\}$ decorated with a $\tb$, then by Definition~\ref{def:g-function}(e), $g(D_w)=D_{\phi(w)}$. 

        \item[(6)] If $w \in (\PZZ)$, then $w$ does not contain factors of the form $s_{n+1}s_n s_{n+1}$. Hence $D_{\phi(w)}$ has as many $D_{n+1}D_{n+2}$ factors as the occurrences of $s_{n+1}$ in $w$, which means as many $\loc$ as the occurrences of $s_{n+1}$. At the same time, $D_w$ has as many $\tb$ as the occurrences of $s_{n}s_{n+1}s_n$ and a certain number of $\circ$ whenever $w$ starts or ends with a $s_{n+1}$. Then again, by Definition~\ref{def:g-function}(f), $g(D_w)=D_{\phi(w)}$.
    \end{itemize}
    The map $\tilde{\theta}_B$ is an isomorphism since $g\circ \tilde{\theta}_B=\tilde{\theta}_D \circ f_\phi$, the maps $g$, $f_\phi$ are injective and $\tilde{\theta}_D$ is an isomorphism by Theorem~\ref{injec}.
\end{proof}

 Arguments similar to those in this section can be developed and used to describe the diagrammatic representation of $\TL(\Cn)$ introduced by Ernst in \cite{ErnstDiagramI, ErnstDiagramII} and to prove its injectivity.

\section{Further developments}\label{sec:ultimo}

There are two possible developments in Kazhdan--Lusztig theory related to this paper.
First, it is known that $\TL(\Gamma)$ admits a special basis called \textit{canonical basis} or \textit{IC basis} for $\Gamma$ of arbitrary type, see \cite{Green-Loson}. We think that the diagram algebras introduced in this paper and in \cite{ErnstDiagramI} could be useful to provide explicit descriptions of such bases in affine types, as in \cite{Green-ICbasis}. 
Moreover, in~\cite{Green-trace} and \cite{Green_TLE}, the author showed that an analogous of the Jones's trace can be used to compute easily the leading coefficient $\mu(x,w)$ of the Kazdhan--Lusztig polynomial when $x$ and $w$ are in $\FC(\Gamma)$. It would be interesting to use the diagram calculus developed in this paper and in \cite{ErnstDiagramI} to construct explicitly such a trace when $\Gamma$ is of affine type.

\section*{Acknowledgements}
We would like to thank Richard Green for his interesting suggestions that gave rise to Section~\ref{Green} and the anonymous referees for their valuable comments. The first author thanks GNSAGA of INdAM for its partial support while visiting the second author.

%%%%%%%%%%%%%%%%%%%%%%%%%%%%%%%
\appendix
\section{Examples of factorizations}\label{A}
%%%%%%%%%%%%%%%%%%%%%%%%%%%%%%%

In what follows we give explicit applications of the cp-operation and the $K_L$-operation defined in Definition~\ref{definition:cutandpaste} and in the proof of Proposition~\ref{unicityP}, respectively. By iterating such operations we recover a factorization for each of the three diagrams. The first label in the pictures indicates the simple diagram associated to the simple edge used at each step while the one in brackets the corresponding local operation in Figure~\ref{cp-general}  that is applied.
\begin{figure}[hbtp]
	\centering
    \includegraphics[width=\textwidth]{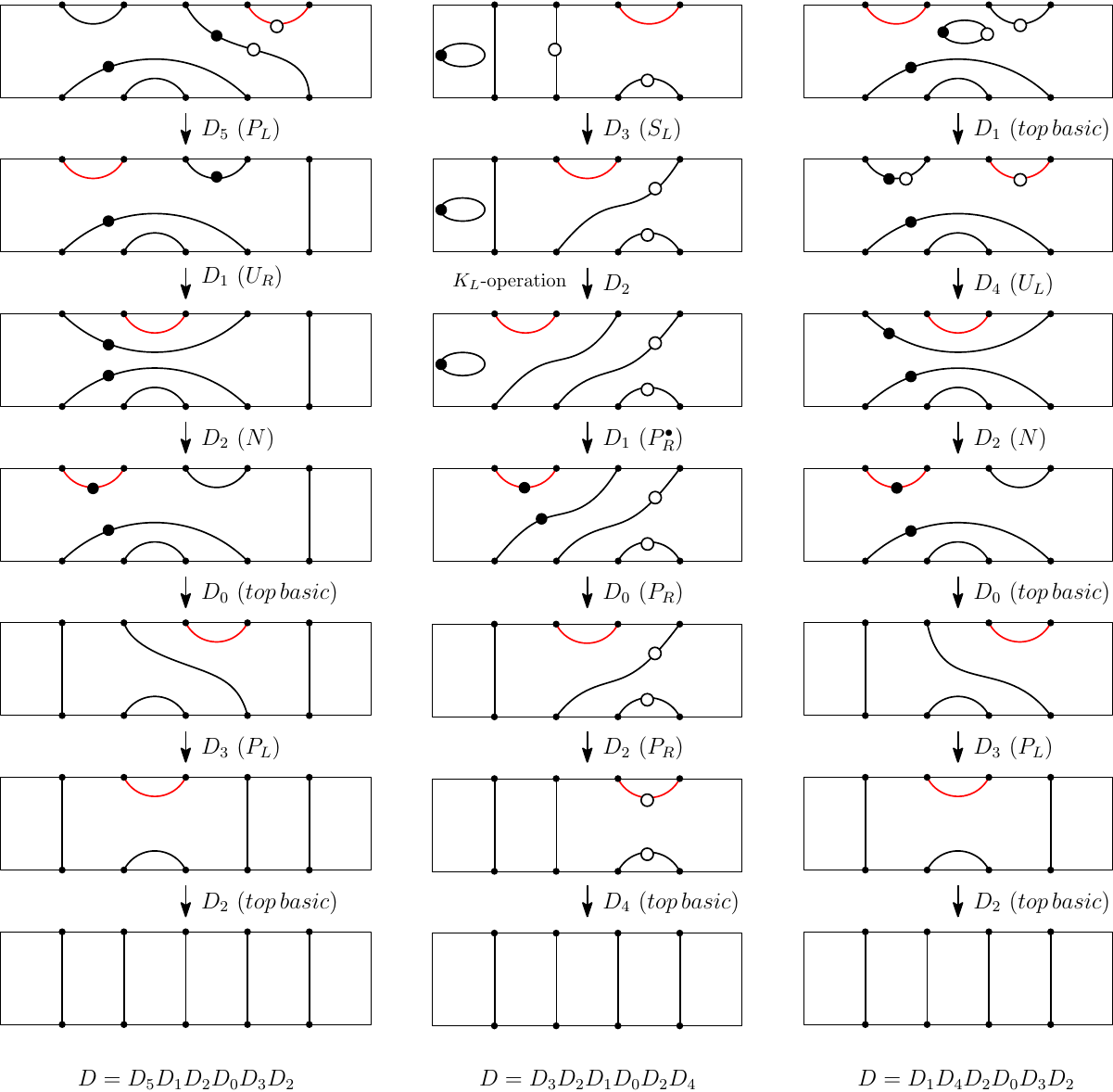} 
	\caption{Examples of cp and $K_L$ operations.}
	\label{appendix}
\end{figure}

\begin{table}[ht]
    \centering
\begin{tabular}{|c|c|l|c|}
\hline
  $e$   & $\bf{d}$ & \quad \quad \quad \quad \quad \quad \quad \quad $\weight$ & $\nu(D)-\nu(\cp_e(D))$\\
     \hline
  $N$   &  $\bullet^\alpha \, \circ^\beta$    & $\weight(j,k)_D=\alpha(j-1)+\beta(n+2-k)$   & 2\\
       &                                 & $\weight(i+1,  k)_{\cp_e(D)}= \beta(n+2-k)$      &        \\
        &                                & $\weight(j,i)_{\cp_e(D)}=\alpha(j-1)$     & \\
       \hline
  $U_R$ &   $\bullet \, \circ^\beta$ & $\weight(i+2, k)_D=i+1+\beta(n+2-k)$ &  -2 \\
        &                       & $\weight(i,k)_{\cp_e(D)}=i-1+\beta(n+2-k)$ & \\
  \hline
  $P_R$  &  &  & 2 \\
        & $\bullet (\circ \, \bullet)^p$ & $\weight(i+2, j')_D=j-1+p(n+1)$ & \\
        &                           & $\weight(i, j')_{\cp_e(D)}=j-1+p(n+1)$ &\\
   \cline{2-3}
   & $\circ (\bullet \, \circ)^p$ & $\weight(i+2, j')_D=n-i+p(n+1)$ &  \\
   &                    &   $\weight(i,j')_{\cp_e(D)}=p(n+1)$, & \\
   &                   & $\weight(i+1, i+2)_{\cp_e(D)}=n-i$ & \\
   \cline{2-3}
    & $(\bullet \, \circ)^{p+1}$ & $\weight(i+2,j')_D=(p+1)(n+1)$ &  \\
    &                       & $\weight(i, j')_{\cp_e(D)}=(p+1)(n+1)$ & \\
    \cline{2-3}
     & $(\circ \, \bullet)^{p+1}$ & $\weight(i+2, j')_D=(p+1)(n+1)-i-2+j$ &  \\
     &                      & $\weight(i+1, i+2)_{\cp_e(D)}=n-i$ & \\
     &                      & $\weight(i, j')_{\cp_e(D)}=j-1+p(n+1)$, &\\
    \hline
   $P_R^\bullet$ &  &  $\weight(\lob)_D=1$ & 0  \\
   \hline
  $S_R$ & & &  -2 \\
   & $\bullet (\circ \, \bullet)^p$ & $\weight(i+2, (i+2)')_D=i+1+p(n+1)$ & \\
   &  & $\weight(i, (i+2)')_{\cp_e(D)}=i-1+p(n+1)$ & \\
   
   \cline{2-3}
   & $\circ (\bullet \, \circ)^p$ & $\weight(i+2, (i+2)')_D=n-i+p(n+1)$ & \\
   &                        & $\weight(i,(i+2)')_{\cp_e(D)}=p(n+1)-2$, &\\
   &                       &  $\weight(i+1, i+2)_{\cp_e(D)}=n-i$   & \\
   \cline{2-3}
    & $(\bullet \, \circ)^{p+1}$ & $\weight(i+2,(i+2)')_D=(p+1)(n+1)$ &  \\
    &                       & $\weight(i, (i+2)')_{\cp_e(D)}=(k+1)(n+1)-2$ & \\
    \cline{2-3}
     & $(\circ \, \bullet)^{p+1}$ & $\weight(i+2, (i+2)')_D=(p+1)(n+1)$ &  \\
     &                      &  $\weight(i, (i+2)')_{\cp_e(D)}=i-1+p(n+1)$, & \\
      &                     &   $\weight(i+1, i+2)_{\cp_e(D)}=n-i$ & \\
    \hline
  \end{tabular}
\caption{Computations for the proof of statement (3) of Lemma~\ref{lemma:cplength}.}
    \label{tabular}
\end{table}

\def\cprime{$'$}

\end{document}